\documentclass[a4paper]{article}

\usepackage{graphicx,amssymb}
\usepackage{amsmath}
\usepackage{cases}
\usepackage{float}
\usepackage{listings}
\usepackage{times}
\usepackage{tikz}
\usepackage{epic}
\usepackage{titlesec}
\usepackage{geometry}
\usepackage{amsthm}
\usepackage{bm}
\usepackage{stmaryrd}
\usepackage{subfigure}
\usepackage{caption}
\usepackage{epstopdf}
\usepackage{hyperref}
\usepackage{indentfirst}
{\numberwithin{equation}{section}
\setlength{\parindent}{1em}
\newtheorem{theorem}{Theorem}[section]
\newtheorem{lemma}{Lemma}[section]
\newtheorem{remark}{Remark}[section]

\newtheorem{example}{Example}[section]

\SetSymbolFont{stmry}{bold}{U}{stmry}{m}{n}
\def\jump#1{\llbracket#1\rrbracket}

\newcommand{\normmm}[1]{{\left\vert\kern-0.25ex\left\vert
\kern-0.25ex\left\vert #1
    \right\vert\kern-0.25ex\right\vert\kern-0.25ex\right\vert}}

\geometry{left=3cm,right=3cm,top=4cm,bottom=2.5cm}

\begin{document}

\captionsetup[figure]{labelfont={bf},name={Fig.},labelsep=period}

\title{A posteriori error estimates for the mortar\\ staggered DG method}
\author{ Lina Zhao\footnotemark[1]\qquad
\;Eric Chung\footnotemark[2]\qquad}
\renewcommand{\thefootnote}{\fnsymbol{footnote}}
\footnotetext[1]{Department of Mathematics, The Chinese University of Hong Kong, Hong Kong Special Administrative Region. ({lzhao@math.cuhk.edu.hk})}
\footnotetext[2]{Department of Mathematics, The Chinese University of Hong Kong, Hong Kong Special Administrative Region. ({tschung@math.cuhk.edu.hk})}

\date{}

\maketitle

\textbf{Abstract:}
Two residual-type error estimators for the mortar staggered discontinuous Galerkin discretizations of second order elliptic equations are developed. Both error estimators are proved to be reliable and efficient. Key to the derivation of the error estimator in potential $L^2$ error is the duality argument. On the other hand, an auxiliary function is defined, making it capable of decomposing the energy error into conforming part and nonconforming part, which can be combined with the well-known Scott-Zhang local quasi-interpolation operator and the mortar discrete formulation yields an error estimator in energy error. Importantly, our analysis for both error estimators does not require any saturation assumptions which are often needed in the literature. Several numerical experiments are presented to confirm our proposed theories.

\textbf{Keywords:} Staggered grids, Discontinuous Galerkin method, Nonmatching grids, A posteriori error estimates, Adaptive mesh refinement

\section{Introduction}

The mortar element method is a domain decomposition method
with non-overlapping subdomains \cite{Bernardi93,Bernardi94}.
One distinctive feature of mortar finite element method is that the meshes on adjacent subdomains are not required to be matching with each other, which makes the method well suited for problems with complicated geometries. Local features of the solution such as corner singularities or large gradients can be resolved by finer grids in the local region. Furthermore, large scale features such as geological faults and layers in subsurface flow can be modeled with nonmatching grids. Staggered discontinuous Galerkin (SDG) methods pioneered by Chung and Engquist \cite{ChungWave, ChungWave2} earn many desirable properties such as mass conservation, superconvergence and the flexibility to deal with general quadrilateral and polygonal meshes, and have been applied to numerous partial differential equations arising from the practical applications (cf. \cite{ChungCockburn16,LeeKim16,ChungQiu17,DCNStokes,ChungParkLina,CheungChungKim18,LinaPark18,LinaParkShin}). To further advance the applications of SDG method, a mortar formulation is developed for SDG method \cite{KimChung16}, where different triangulations in different regions of the computational domain are exploited. In the framework proposed therein, SDG discretization is employed in each subdomain and
the continuity of the solution across the subdomain interfaces
is imposed through the introduction of the Lagrange multipliers.
The analysis developed therein shows that optimal convergence rates in both $L^2$ and discrete energy norms are achieved. In addition, the numerical results there illustrate that if the
exact solution earns local singularities, one can only obtain optimal convergence rates in regularity, not in rate. To efficiently capture the singularities and achieve optimal approximation with minimum degrees of freedom, adaptive finite element method based on a posteriori error estimators can be utilized. Due to the nonmatching meshes across the subdomain interfaces, mortar finite element methods are favored for adaptive mesh refinement. Indeed, nonmatching grids can be used on the different subdomains of a partition, this can highly reduce the number of degrees of freedom since no further nodes must be added to avoid the nonconforming meshes on the subdomain interfaces.

A posteriori error estimators have been actively studied for
mixed finite element methods and discontinuous Galerkin methods on conforming grids \cite{Braess96,Verf96,Carsten97,KPDG,Ainsworth07,kim07, LarsonAxel08,KimPark08,martin10,kimpark-sinum10,ckp11,ErnVohralik15,cg16,DuChung18,LinaPark19}
since the pioneering work of Babu\v{s}ka and Rheinboldt \cite{Babuska78, Babuska782}. However, a posteriori error
analysis for the discretization problems on nonmatching grids
is still a largely undeveloped area. Wohlmuth introduces residual type and hierarchical type a posteriori error estimators in \cite{Wohlmuth99,Wohlmuth992} for mortar finite element methods. Wheeler and Yotov \cite{Wheeler05} propose two types of a posteriori error estimators for the mortar mixed finite element method. All these estimators are constructed with some saturation assumptions. To avoid this assumption, Bernardi and Hecht present some residual type error estimators without the presence of saturation assumptions \cite{Bernardi02}. But the mesh nodes are required to be coinciding on the interface. To extend the current framework of a posteriori error analysis developed on matching grids to nonmatching grids in a more general sense, it is important to exclude the mesh restrictions on the mortar and non-mortar sides, and avoid the saturation assumptions. Recently, some residual type a posteriori error estimators are developed based on the posteriori version of the well known Strang Lemma \cite{WangXu}, where the aforementioned restrictions are remitted.

The purpose of this paper is to design and analyze two residual type error estimators for mortar SDG method without any saturation assumptions. We first derive a reliable and efficient error estimator for mortar SDG method in potential $L^2$ error, where
the key ingredient is the duality argument. In contrast to the general error estimators developed for conforming grids, the jump of solution across the subdomain interfaces and a mortar flux difference term are also involved.
Then we propose an error estimator in
energy error, where some difficulties arise due to the following aspects: First, the interface grids are nonmatching across the adjacent subdomains, defining related conforming finite element spaces on the non-matching meshes as the methodology exploited in \cite{ch07} is impossible; Second, mortar SDG method is a mixed type method, applying the Galerkin orthogonality directly like those proposed in \cite{WangXu} is infeasible. To overcome the aforementioned difficulties, we employ the well-known Scott-Zhang local quasi-interpolation operator defined in \cite{ScottZhang90}, which is conforming in each subdomain avoiding constructing a conforming operator over the whole domain that is usually cumbersome. Then we define an auxiliary function $s\in H^1_0(\Omega)$, which enables us to decompose the energy error into conforming part and nonconforming part. Combing the above primary ingredients, the error estimator in energy error can be derived. Again, in addition to element residual terms, the jump of solution across the subdomain interfaces and the mortar flux difference terms are also involved, where the presence of the additional terms is due to the mortar matching condition. We emphasize that our analysis for both error estimators does not need any mesh restrictions on the mortar and non-mortar sides, and saturation assumptions are also avoided.


The rest of the paper is organized as follows. In the next section, we briefly introduce the mortar formulation of SDG method. In Section~\ref{reliability}, two residual type error estimators are proposed, and the reliability of the proposed error estimators are proved. Then, the efficiency of the proposed error estimators are established in Section~\ref{Sec:efficiency}. Several numerical experiments are carried out in Section~\ref{Sec:numerical}, where the performances of the two error estimators are displayed. Numerical results demonstrate
that singularities can be well captured and optimal convergence rates can be achieved under the adaptive mesh refinement.
Finally, some conclusions are given at the end of this paper.

\section{Mortar formulation of SDG method}
In this section, we briefly describe the mortar formulation of SDG method by following the framework developed in \cite{KimChung16}. The primary ingredient is to impose the continuity of the solution across subdomain interfaces by a mortar matching condition. To begin, we consider the following second order elliptic problem in two dimensions:
\begin{equation}
\begin{split}
-\nabla \cdot (\rho \nabla u)&=f \quad \mbox{in}\; \Omega,\\
u&=0 \quad \mbox{on}\; \partial \Omega,
\end{split}
\label{model}
\end{equation}
where $\Omega$ is the computational domain and $f(x)$ is a given source function. We divide the domain $\Omega$ into a set of $N$ non-overlapping subdomains, $\bar{\Omega}=\cup_{i=1}^N\bar{\Omega}_i$. We assume, for simplicity, that $\{\Omega_i\}_{i=1}^N$ is a geometrically conforming partition of $\Omega$. We further assume that $\rho$ is a piecewise constant function, which equals $\rho_i$ in $\Omega_i$. Every subdomain $\Omega_i$ is equipped with a quasi-uniform triangulation $\mathcal{T}_{h_i}$ with mesh size $h_i>0$. The triangulations $\{\mathcal{T}_{h_i}\}_{i=1}^N$ can be non-matching across the subdomain interface $\Gamma=\cup_{i,j=1}^N\Gamma_{ij}$, where $\Gamma_{ij}(=\partial \Omega_i\cap \partial \Omega_j)$ is the interface shared by the two subdomains $\Omega_i$ and $\Omega_j$. In addition, $S=\{(i,j): \Omega_i\, \mbox{and}\, \Omega_j\, \mbox{have non-empty intersection}\}$. In addition, we define $\Gamma_i=\partial \Omega_i\cap \Gamma=\partial \Omega_i\backslash \partial \Omega$.

Let $D\subset \mathbb{R}^d, d=1,2$, we adopt the standard notations for
the Sobolev spaces $H^s(D)$ and their associated norms $\|\cdot\|_{s,D}$, and semi-norms $|\cdot|_{s,D}$ for $s\geq 0$. The space $H^0(D)$ coincides with $L^2(D)$, for which the norm
and inner products are denoted as $\|\cdot\|_D$ and $(\cdot,\cdot)_D$, respectively. If $D=\Omega$, the subscript $\Omega$ will be dropped unless otherwise mentioned. In the sequel, we use $C$ to denote a generic positive constant independent of the meshsize which can have different values at different occurrences.

Next, we define some spaces which will be utilized later
\begin{align*}
Q_i=L^2(\Omega_i)^2,\quad Q=\prod_{i=1}^NQ_i,\quad M=L^2(\Gamma)
\end{align*}
and
\begin{align*}
V_i=\{v\in H^1(\Omega_i), v=0\;\mbox{on}\; \partial \Omega_i\cap \partial \Omega\},\quad V=\prod_{i=1}^N V_i.
\end{align*}
We rewrite (\ref{model}) into a first order system by introducing an additional unknown $\bm{z}$
\begin{equation*}
\begin{split}
\bm{z}&=\rho \nabla u\hspace{0.45cm} \mbox{in}\;\Omega,\\
-\nabla \cdot\bm{z}&=f\hspace{0.9cm}\mbox{in}\;\Omega,\\
u&=0 \hspace{0.9cm}\mbox{on}\;\partial \Omega,
\end{split}
\end{equation*}
which can be recast into the equivalent subdomain problem
\begin{equation}
\begin{split}
\bm{z}_i&=\rho_i \nabla u_i \hspace{0.3cm} \mbox{in}\; \Omega_i,\\
-\nabla \cdot \bm{z}_i&=f \hspace{0.6cm}\quad \mbox{in} \;\Omega_i,\\
\bm{z}_i\cdot \bm{n}_i&=\lambda_i \hspace{0.8cm}\mbox{on}\; \partial \Omega_i\backslash \partial \Omega,\\
u_i&=0 \hspace{0.6cm}\quad \mbox{on}\; \partial \Omega_i \cap \partial \Omega,\\
\lambda_i+\lambda_j&=0 \hspace{0.9cm}\mbox{for all} \;(i,j)\in S
\end{split}
\label{subdomain-system}
\end{equation}
with an additional condition that $u_i$ are continuous across the subdomain interface. Let $\bm{n}_{ij}$ be the fixed unit normal direction on $\Gamma_{ij}$ common to the two subdomains $\Omega_i$ and $\Omega_j$. We define $\lambda$ on $\Gamma$ as
\begin{align*}
\lambda\mid_{\Gamma_{ij}}=\lambda_i\, \bm{n}_i\cdot \bm{n}_{ij}=\lambda_j\,\bm{n}_j\cdot\bm{n}_{ij} \quad \forall\, \Gamma_{ij}\subset \Gamma.
\end{align*}

Multiplying the equations in \eqref{subdomain-system} by the corresponding test functions and integration by parts, we obtain the weak formulation: find $(u,\bm{z},\lambda)\in V\times Q \times M$ such that
\begin{equation}
\begin{split}
\rho_i^{-1}(\bm{z}, \bm{q})_{\Omega_i}&=(\nabla u, \bm{q})_{\Omega_i}\quad \forall \bm{q}\in Q_i,\\
(\bm{z}, \nabla v)_{\Omega_i}-(\lambda \bm{n}_{ij}\cdot\bm{n}_i, v \, )_{\Gamma_i}&=(f,v)_{\Omega_i}\qquad \forall v\in V_i,\\
\sum_{i=1}^N(\jump{u}, \mu )_{\Gamma_i}&=0 \hspace{1.7cm} \forall \mu\in M,
\end{split}
\label{weak}
\end{equation}
where $\lambda \,\bm{n}_{ij}\cdot\bm{n}_i=\bm{z}\mid_{\Omega_i}\cdot \bm{n}_i$.

We can rewrite \eqref{weak} as: find $(\bm{z}, u, \lambda)\in Q\times V
\times M$ such that
\begin{align*}
A(\bm{z}, u, \lambda; \bm{q}, v, \mu)=(f,v)\quad \forall  ( \bm{q}, v, \mu)\in Q\times V\times M,
\end{align*}
where
\begin{equation*}
\begin{split}
A(\bm{z}, u, \lambda; \bm{q}, v, \mu)=\sum_{i=1}^N\Big(\rho_i^{-1}(\bm{z}, \bm{q})_{\Omega_i}-(\nabla u, \bm{q})_{\Omega_i}+(\bm{z}, \nabla v)_{\Omega_i}-(\lambda, v \,\bm{n}_i\cdot \bm{n}_{ij})_{\Gamma_i}+ (u\bm{n}_i\cdot \bm{n}_{ij}, \mu)_{\Gamma_i}\Big).
\end{split}
\end{equation*}

We then present the construction of the SDG spaces for each $\Omega_i$, and the construction follows the framework given in \cite{KimChung16}. To this end, we first introduce some notations that will be employed later. We let $\mathcal{F}_{u,i}$ be the set of edges in the initial triangulation $\mathcal{T}_{h_i}$ excluding the edges on the interface and $\mathcal{F}_{u,i}^0\subset \mathcal{F}_{u,i}$ be the set of interior edges. For each triangle $\tau\in \mathcal{T}_{h_i}$, we divide it into three subtriangles by connecting an interior point to the three vertices. We note that the interior point can be chosen as the centroid of the triangle to get a good regularity
of the subdivided triangulation.

We denote by $\mathcal{T}_i$ the resulting finer triangulation and by $\mathcal{F}_{p,i}$ the set of edges generated by the subdivision process. In addition, we let $\mathcal{T}=\cup_{i=1}^N\mathcal{T}_{i}$, $\mathcal{T}_h=\cup_{i=1}^N \mathcal{T}_{h_i}$, $\mathcal{F}_p=\cup_{i=1}^N\mathcal{F}_{p,i}$, $\mathcal{F}_u=\cup_{i=1}^N \mathcal{F}_{u,i}$, $\mathcal{F}_u^0=\cup_{i=1}^N \mathcal{F}_{u,i}^0$, $\mathcal{F}^0=\mathcal{F}_u^0\cup\mathcal{F}_p$, and $\mathcal{F}=\mathcal{F}_u\cup\mathcal{F}_p$. We use $h_\tau$ to denote the diameter of $\tau\in \mathcal{T}$, $h_e$ to denote the length of edge $e$, and $h=\max_{\tau\in \mathcal{T}}h_\tau$.

For each edge $e$, we define
a unit normal vector $\bm{n}_{e}$ as follows: If $e\in \mathcal{F}\setminus \mathcal{F}^{0}$, then
$\bm{n}_{e}$ is the unit normal vector of $e$ pointing towards the outside of $\Omega$. If $e\in \mathcal{F}^{0}$, an
interior edge, we then fix $\bm{n}_{e}$ as one of the two possible unit normal vectors on $e$.
When there is no ambiguity,
we use $\bm{n}$ instead of $\bm{n}_{e}$ to simplify the notation.

Let $k\geq 0$ be the order of polynomial used for the approximation and $P^k(\tau)$ be the set of polynomials with degree less than or equal to $k$ defined on $\tau$. We define the following spaces
\begin{align*}
Q_{h_i}=\{\bm{q}:\bm{q}\mid_{\tau}\in P^k(\tau)^2, \tau\in \mathcal{T}_i \; \mbox{and}\; \jump{\bm{q}\cdot\bm{n}}\mid_e=0, \;\forall e\in \mathcal{F}_{p,i}\}
\end{align*}
and
\begin{align*}
V_{h_i}=\{v:v\mid_\tau\in P^k(\tau), \tau\in \mathcal{T}_i \;\mbox{and}\; \jump{v}\mid_e=0,\; \forall e\in \mathcal{F}_{u,i}\},
\end{align*}
where the jumps $\jump{\bm{q}\cdot\bm{n}}\mid_e$ and $\jump{v}\mid_e$ are defined in the standard way
\begin{align*}
\jump{\bm{q}\cdot\bm{n}}\mid_e:=\bm{q}\mid_{\tau_1}\cdot\bm{n}-
\bm{q}\mid_{\tau_2}\cdot\bm{n}\quad \mbox{and}\quad \jump{v}\mid_e:=v\mid_{\tau_1}-v\mid_{\tau_2}.
\end{align*}
In the above, $\tau_1$ and $\tau_2$ are the two triangles with the common edge $e$. In the above definition, we assume $\bm{n}$ is pointing from $\tau_1$ to $\tau_2$. In addition, we define
\begin{align*}
V_{h_i}^0=\{v\in V_{h_i}: v=0\;\mbox{on}\; \partial \Omega\cap \partial \Omega_i\}.
\end{align*}
On the whole computational domain, we define $V_h:=\prod_{i=1}^NV_{h_i}^0$ and $Q_h:=\prod_{i=1}^N Q_{h_i}$.

\begin{figure}[H]
\centering
\includegraphics[width=6cm]{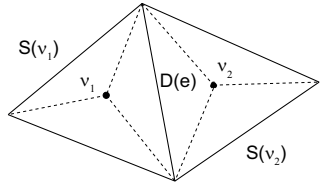}
\caption{Region $S(\nu)$ (initial triangle) and Region $D(e)$ (dotted quadrilateral): $\mathcal{F}_u$ (initial edges, solid line) and $\mathcal{F}_p$ (new edges, dotted line) .}
\label{mesh}
\end{figure}


\begin{figure}[H]
\setlength{\abovecaptionskip}{-26pt}
\centering
\includegraphics[width=6cm]{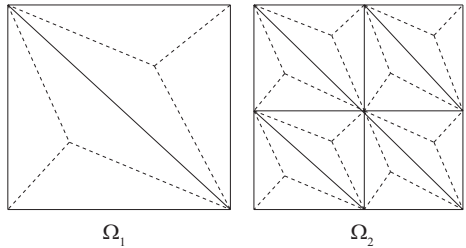}
\caption{Nonmatching initial triangulation in two neighbouring subdomains}
\label{initial}
\end{figure}

We recall that $\Gamma_{ij}$ is the interface between $\Omega_i$ and $\Omega_j$, see {\sc Fig}.~\ref{initial}. On $\Gamma_{ij}$,
we introduce two different meshes called $T_{ij,i}$ and $T_{ij,j}$, which are respectively defined as the restrictions
of $\mathcal{T}_{h_i}$ and $\mathcal{T}_{h_j}$ on $\Gamma_{ij}$. Among these two meshes, we select one as non-mortar mesh and the other as mortar mesh. For the non-mortar mesh, say $T_{ij,i}$, we introduce the space of Lagrange multipliers $M_{ij}=V_{h_i}^0\mid_{\Gamma_{ij}}$, which consists of piecewise polynomials of degree up to $k$ defined on $\Gamma_{ij}$ with respect to the mesh $T_{ij,i}$. Also, we denote the union of all the non-mortar mesh as $\mathcal{T}^{\Gamma,h}=\cup_{i,j=1}^N T_{ij,i}$. In addition, we define $M_h:=\Pi_{(i,j)\in S} M_{ij}$. The space $M_h$ is used to enforce continuity of functions in $V_h$. Specially, we define the following mortar SDG space for the approximation of $u$
\begin{align*}
\widehat{V}_h=\{v=(v_1, \ldots, v_N)\in V_h: \int_{\Gamma_{ij}} (v_i-v_j)\psi\;ds=0 \quad \forall \psi\in M_{ij},\; \forall\, \Gamma_{ij}\subset \Gamma\}.
\end{align*}
Furthermore, we use $S(\nu)$ to denote the triangle in the initial triangulation $\mathcal{T}_{h_i}$ with $\nu$ denoting the interior point chosen in the above subdivision process. Thus, $S(\nu)$ is the union of the three triangles in $\mathcal{T}_i$ having the interior point $\nu$ as a common vertex. For an edge $e\in \mathcal{F}_{u,i}^0$, we let $D(e)$ be the union of the two triangles in $\mathcal{T}_i$ sharing the edge $e$, and for an
edge $e\in \mathcal{F}_{u,i}\cap \partial \Omega_i$, we let $D(e)$ be the triangle in $\mathcal{T}_i$ having the edge $e$, see {\sc Fig}.~\ref{mesh} for an illustration. In addition for $e\in \mathcal{T}^{\Gamma,h}$, we use $D(e)$ to denote the union of the simplicial submeshes on both sides sharing the edge $e$ or part of $e$.


To derive the discrete version for \eqref{subdomain-system}, we introduce $\lambda_h\in M_h$ to approximate $\lambda$. We note that $\lambda_h\mid_{\Gamma_{ij}}$ is considered as an approximation of the flux $\bm{z}\cdot \bm{n}_{ij}$ on $\Gamma_{ij}$.
Following \cite{KimChung16}, we define
\begin{align*}
c_i(v,\mu)=\sum_{\Gamma_{ij}\subset \partial \Omega_i}\int_{\Gamma_{ij}}v\,\mu\, \bm{n}_i\cdot \bm{n}_{ij}\;ds \quad \forall (v,\mu)\in V_{h_i}^0\times \prod_{(i,j)\in S}M_{ij}.
\end{align*}
We also define the following bilinear forms
\begin{align*}
b_i(\bm{z},v)&=( \bm{z}, \nabla v)_{\Omega_i}-\sum_{e\in \mathcal{F}_{p,i}}( \bm{z}\cdot \bm{n},\jump{v})_e,\\
b_i^*(v,\bm{q})&=-(v,\nabla \cdot \bm{q})_{\Omega_i}+\sum_{e\in \mathcal{F}_{u,i}^0} (v,\jump{\bm{q}\cdot \bm{n}})_e+(v,\bm{q}\cdot \bm{n}_i)_{\Gamma_i},
\end{align*}
where the gradient and divergence operators are elementwise operators. Integration by parts reveals that the above bilinear forms are adjoint to each other, namely, $b_i(\bm{q}, v)=b_i^*(v, \bm{q})$\; $\forall (v,\bm{q})\in V_{h_i}^0\times Q_{h_i}$.

With the aforementioned ingredients, the mortar SDG discretization for \eqref{subdomain-system} reads: find $(u_h, \bm{z}_h, \lambda_h)\in V_h\times Q_h\times M_h$ such that
\begin{equation}
\begin{split}
\rho_i^{-1}(\bm{z}_h, \bm{q})_{\Omega_i}&=b_i^*(u_h, \bm{q})\quad\; \forall \bm{q}\in Q_{h_i}, \; i=1,\ldots, N,\\
b_i(\bm{z}_h, v)-c_i(v, \lambda_h)&=(f,v)_{\Omega_i}\qquad \forall v\in V_{h_i}^0, \; i=1,\ldots, N,\\
\sum_{i=1}^Nc_i(u_h, \mu)&=0 \hspace{1.7cm} \forall  \mu\in M_h,
\end{split}
\label{discrete}
\end{equation}
which can be rewritten as: find $(u_h, \bm{z}_h, \lambda_h)\in V_h\times Q_h\times M_h$ such that
\begin{align*}
A_h(\bm{z}_h, u_h, \lambda_h; \bm{q}, v, \mu)=(f,v) \quad \forall (v,\bm{q}, \mu)\in V_h\times Q_h\times M_h,
\end{align*}
where
\begin{align}
A_h(\bm{z}_h, u_h, \lambda_h; \bm{q}, v, \mu)=\sum_{i=1}^N\Big(\rho_i^{-1}(\bm{z}, \bm{q})_{\Omega_i}-b_i^*(u_h, \bm{q})+b_i(\bm{z}_h, v)-c_i(v, \lambda_h)+c_i(u_h, \mu)\Big).\label{global-discrete}
\end{align}
We infer from integration by parts
\begin{align}
A_h(\bm{z}-\bm{z}_h, u-u_h, \lambda-\lambda_h; \bm{q}, v, \mu)=0 \quad \forall (v,\bm{q}, \mu)\in V_h\times Q_h\times M_h.\label{consistency}
\end{align}

We recall some a priori error estimates from \cite{KimChung16} which are needed later to illustrate the efficiency of the proposed error estimators.
\begin{lemma}\label{priori-estimate}
Let $u|_{\Omega_i}\in H^{1+\sigma_i}(\Omega_i)$ with $\sigma_i>1/2$. Let $(\bm{z}_h,u_h, \lambda_h)\in Q_h\times V_h\times M_h$ be the solution of (\ref{discrete}). Then the following estimates hold
\begin{align*}
\sum_{i=1}^N\rho_i\|\nabla (u-u_h)\|_{0,\Omega_i}^2&\leq C \sum_{i=1}^N\rho_ih_i^{2\min\{\sigma_i,k\}}
\|u\|_{\sigma_i+1,\Omega_i}^2,\\
\sum_{i=1}^N\rho_i\|u-u_h\|_{0,\Omega_i}^2&\leq C \sum_{i=1}^N\rho_ih_i^{2\min\{\sigma_i,k\}+2}\|u\|_{\sigma_i+1,\Omega_i}^2,\\
\sum_{i=1}^N \rho_i^{-1}\|\bm{z}-\bm{z}_h\|_{0,\Omega_i}^2&\leq C \sum_{i=1}^N\rho_ih_i^{2\min\{\sigma_i,k+\frac{1}{2}\}}
\|u\|_{\sigma_i+1,\Omega_i}^2.
\end{align*}
\end{lemma}

\section{Reliability}\label{reliability}
In this section, two residual-type error estimators are proposed. First, we develop an error estimator in potential $L^2$ error, which mainly relies on the duality argument. Next, we propose an energy error estimator based on an auxiliary function and the well-known Scott-Zhang local quasi-interpolation operator.

\subsection{Potential error estimator}
To begin, we recall the following trace inequality
\begin{align}
\|v\|_{0,e}\leq C(h_\tau^{-\frac{1}{2}}\|v\|_{0,\tau}+h_\tau^{\frac{1}{2}}\|\nabla v\|_{0,\tau}) \quad \forall v\in H^1(\tau), \quad \forall \tau\in \mathcal{T},  \;e\in \partial \tau.\label{trace}
\end{align}
and
\begin{align}
\|v\|_{1/2,e}\leq C \|v\|_{1,\tau} \quad \forall v\in H^1(\tau), \tau\in \mathcal{T}, e\in \partial \tau. \label{trace2}
\end{align}
We then define two interpolation operators $I_i:H^1(\Omega_i)\rightarrow V_{h_i}$ and $J_i: H^{\epsilon}(\Omega_i)^2\rightarrow Q_{h_i},\epsilon>1/2$ by
\begin{align*}
(I_i v, \phi)_\tau&=(v, \phi)_\tau \quad \forall \phi\in P^{k-1}(\tau), \;\tau\in \mathcal{T},\\
(I_iv, \phi)_e&=(v,\phi)_e\quad\; \forall \phi\in P^k(e),\;e\in \mathcal{F}_{u}
\end{align*}
and
\begin{align*}
(J_i\bm{\tau},\bm{q})_\tau&=(\bm{\tau},\bm{q})_\tau \qquad\;\forall \bm{q}\in P^{k-1}(\tau)^2,\tau\in \mathcal{T},\\
(J_i \bm{\tau}\cdot \bm{n}_e, \phi)_e&=(\bm{\tau}\cdot \bm{n}_e, \phi)_e \quad \forall \phi\in P^k(e),\; e\in \mathcal{F}_{p}.
\end{align*}

In addition, we let $\pi_e$ be the $L^2$ projection operator onto $M_h$. Then, the following inequalities hold true for smooth enough functions $\bm{q}, v$ and $\mu$ (cf. \cite{Ciarlet78,KimChung16})
\begin{equation}
\begin{split}
\|\mu-\pi_e\mu\|_{0,e}&\leq C h_e^{1/2}\|\mu\|_{1/2,e},\\
\|v-I_iv\|_{0,\tau}&\leq C h_\tau^{l+1}\|v\|_{l+1,\tau},\quad l=0,1,\\
\|v-I_iv\|_{0,e}&\leq C h_e^{3/2}\|v\|_{2,\tau},\\
\|\bm{q}-J_i\bm{q}\|_{0,\tau}&\leq C h_\tau\|\bm{q}\|_{1,\tau},\\
\|\bm{q}-J_i \bm{q}\|_{0,e}&\leq C h_\tau^{1/2}\|\bm{q}\|_{1,\tau}.
\end{split}
\label{approximation}
\end{equation}
On each element $\tau\in \mathcal{T}$, we define the local error estimator as
\begin{align*}
\eta_\tau^2&=h_\tau^4\|f+\nabla \cdot \bm{z}_h\|_{0,\tau}^2+h_\tau^2\|\rho^{-1}\bm{z}_h-\nabla u_h\|_{0,\tau}^2+\sum_{e\in \partial \tau\cap \mathcal{F}_{u}^0}h_e^3\|\jump{\bm{z}_h\cdot \bm{n}}\|_{0,e}^2\\
&\;+\sum_{e\in \partial \tau\cap \Gamma}h_e^3\|\lambda_h\bm{n}_i\cdot \bm{n}_{ij}-\bm{z}_h\cdot \bm{n}_i\|_{0,e}^2+\sum_{e\in \partial \tau \cap (\mathcal{T}^{\Gamma, h}\cup \mathcal{F}_p)}h_e\|\jump{u_h}\|_{0,e}^2.
\end{align*}
Then the global error estimator in potential $L^2$ error can be defined as
\begin{align}
\eta_1^2=\sum_{\tau\in \mathcal{T}}\eta_{\tau}^2.\label{eta-1}
\end{align}
The main result of this section can be stated in the next theorem.
\begin{theorem}\label{thm:eta1}
Let $(\bm{z}, u)$ be the weak solution of \eqref{weak} and $(\bm{z}_h,u_h, \lambda_h)\in Q_h\times V_h\times M_h$ be the solution of \eqref{discrete}. Let $\eta_1$ be defined in \eqref{eta-1}, then there exists a positive constant $C$ such that
\begin{align*}
\|u-u_h\|_0\leq C \eta_1.
\end{align*}
\end{theorem}

\begin{proof}
Assume the auxiliary problem
\begin{equation}
\begin{split}
-\nabla \cdot (\rho\nabla w)&=u-u_h\quad \mbox{in}\; \Omega,\\
w&=0\hspace{1.2cm} \mbox{on}\; \partial \Omega
\end{split}
\label{duality}
\end{equation}
satisfies the elliptic regularity estimate
\begin{align}
\|w\|_2\leq C \|u-u_h\|_0.\label{regularity}
\end{align}
Let $\bm{U}=-\rho \nabla w$ and $\bm{U}_i\cdot \bm{n}_i=\mu_i$ on $\partial \Omega_i\setminus \partial \Omega$, then
\eqref{duality} can be recast into the following first order system
\begin{align*}
\bm{U}&=-\rho \nabla w\quad \mbox{in}\; \Omega,\\
\nabla \cdot \bm{U}&=u-u_h\quad \mbox{in}\; \Omega,\\
w&=0\hspace{1.2cm}\mbox{on}\; \partial \Omega.
\end{align*}
Notice that the above problem is equivalent to the following subdomain problems
\begin{equation}
\begin{split}
\bm{U}_i&=-\rho_i\nabla w_i \;\; \mbox{in}\; \Omega_i,\\
\nabla \cdot \bm{U}_i&=u-u_h\quad\;\mbox{in}\; \Omega_i,\\
\bm{U}_i\cdot \bm{n}_i&=\mu_i\hspace{1.1cm} \mbox{on} \;\partial \Omega_i\backslash \partial \Omega,\\
w_i&=0 \hspace{1.3cm} \mbox{on}\; \partial \Omega_i\cap \partial \Omega,\\
\mu_i+\mu_j&=0 \hspace{1.3cm}\mbox{for all }\; (i,j)\in S,
\end{split}
\label{sub-domain}
\end{equation}
with the additional condition that $w_i$ are continuous across the subdomain interfaces.

Multiplying the first equation of \eqref{sub-domain} by $\bm{z}-\bm{z}_h$, the second equation by $u-u_h$ and integrating over $\Omega$ to get
\begin{align*}
\|u-u_h\|_0^2&=\sum_{i=1}^N\Big((\nabla\cdot \bm{U}, u-u_h)_{\Omega_i}+\rho_i^{-1}(\bm{U}, \bm{z}-\bm{z}_h)_{\Omega_i}+(\nabla w, \bm{z}-\bm{z}_h)_{\Omega_i}\\
&\;-(w, (\lambda-\lambda_h)\bm{n}_i\cdot \bm{n}_{ij})_{\Gamma_i}\Big).
\end{align*}
Integration by parts, employing the definition of $A_h$ (cf. \eqref{global-discrete}) and using $\bm{U}_i\cdot \bm{n}_i=\mu_i$ on $\partial \Omega_i\backslash \partial \Omega$ yield
\begin{align*}
\|u-u_h\|_0^2&=\sum_{i=1}^N\Big((\bm{U}\cdot \bm{n}_i, u-u_h)_{\Gamma_i}-\sum_{e\in \mathcal{F}_{p,i}} (\bm{U}\cdot \bm{n}, \jump{u_h})_e-(\bm{U}, \nabla (u-u_h))_{\Omega_i}+\rho_i^{-1}(\bm{U}, \bm{z}-\bm{z}_h)_{\Omega_i}\\
&\;+(\nabla w, \bm{z}-\bm{z}_h)_{\Omega_i}-(w, (\lambda-\lambda_h)\bm{n}_i\cdot \bm{n}_{ij})_{\Gamma_i}\Big)\\
&=A_h(\bm{z}, u, \lambda;\bm{U}, w, \mu)-A_h(\bm{z}_h, u_h, \lambda_h; \bm{U}, w, \mu)\\
&=A_h(\bm{z}-\bm{z}_h, u-u_h, \lambda-\lambda_h;\bm{U},w, \mu).
\end{align*}
Employing \eqref{global-discrete} and \eqref{consistency}, we deduce that
\begin{align*}
\|u-u_h\|_0^2&=A_h(\bm{z}-\bm{z}_h, u-u_h, \lambda-\lambda_h;\bm{U}-J_i\bm{U},w-I_iw, \mu-\pi_e \mu)\\
&=\sum_{\tau\in \mathcal{T}}\Big((f+\nabla \cdot \bm{z}_h, w-I_iw)_\tau-(\rho^{-1}\bm{z}_h-\nabla u_h, \bm{U}-J_i\bm{U})_\tau\\
&\;-\sum_{e\in \mathcal{F}_{p}} ( \jump{u_h},(\bm{U}-J_i\bm{U})\cdot \bm{n})_e-\sum_{e\in \mathcal{F}_{u}^0}(\jump{\bm{z}_h\cdot \bm{n}},w-I_iw)_e\\
&\;+ \sum_{i=1}^N\Big((\lambda_h\bm{n}_i\cdot \bm{n}_{ij}-\bm{z}_h\cdot \bm{n}_i,w-I_iw)_{\Gamma_i}- (u_h\bm{n}_i\cdot \bm{n}_{ij}, \mu-\pi_e \mu)_{\Gamma_i}\Big).
\end{align*}
The Cauchy-Schwarz inequality and the approximation properties \eqref{approximation} imply
\begin{align*}
\|u-u_h\|_0^2&\leq C\Big( \sum_{\tau\in \mathcal{T}}
(h_\tau^2\|f+\nabla \cdot \bm{z}_h\|_{0,\tau}\|w\|_{2,\tau}+h_\tau\|\rho^{-1}\bm{z}_h-\nabla u_h\|_{0,\tau}\|\bm{U}\|_{1,\tau}\\
&\;+\|\lambda_h\bm{n}_{i}\cdot\bm{n}_{ij}-
\bm{z}_h\cdot\bm{n}_i\|_{0,\partial \tau\cap \Gamma}h_e^{3/2}\|w\|_{2,\tau})+\sum_{e\in \mathcal{F}_p}h_e^{1/2}\|\jump{u_h}\|_{0,e}\|\bm{U}\|_{1,D(e)}\\
&\;+\sum_{e\in \mathcal{F}_u^0}h_e^{3/2}
\|\jump{\bm{z}_h\cdot\bm{n}}\|_{0,e}\|w\|_{2,D(e)}+\sum_{e\in \mathcal{T}^{\Gamma,h}}h_e^{1/2}\|\jump{u_h}\|_{0,e}
\|\mu\|_{1/2,e}\Big).
\end{align*}
This together with \eqref{trace2} and the elliptic regularity estimate \eqref{regularity} completes the proof.

\end{proof}

\subsection{Energy error estimator}

This section is devoted to the construction of the error
estimator in energy error, the primary ingredient is to define an auxiliary function which enables us to decompose the error into conforming part and nonconforming part. To this end, we first introduce the following lemma, which provides the upper bound for the nonconforming error.

Following Lemma~3.6 of \cite{WangXu} and Theorem~2.2 of \cite{KPDG}, we have
\begin{lemma}\label{lemma:nonconforming}
There exists a positive constant $C$ independent of the mesh size such that
\begin{align*}
\min_{v\in H^1_0(\Omega)}\|\rho^{\frac{1}{2}}\nabla (v-u_h)\|_0\leq C \Big(\sum_{\mathcal{F}_p\cup \mathcal{T}^{\Gamma,h} } h_e^{-1}\|\jump{\rho^{\frac{1}{2}}u_h}\|_{0,e}^2\Big)^{1/2}.
\end{align*}

\end{lemma}

Let $\Pi_h: H^1(\Omega_i)\rightarrow X_h(\Omega_i)$ be the well-known Scott-Zhang local quasi-interpolation operator, where $X_h(\Omega_i)$ is $P^1$ conforming element space in each subdomain $\Omega_i$. In addition, $\Pi_h$ satisfies the following interpolation error estimates (cf. \cite{ScottZhang90}).
\begin{lemma}\label{lemma:error-Ih}
We have the following interpolation error estimates for $v\in H^1_0(\Omega)$
\begin{align*}
h_{\tau}^{-2}\|v-\Pi_hv\|_{0,\tau}^2&\leq C |v|_{1,\omega_\tau} \quad \tau\in\mathcal{T}_i,\\
h_e^{-1}\|v-\Pi_h v\|_{0,e}^2&\leq C |v|_{1,\omega_e}^2\quad e\in \mathcal{F}_{u,i}\cup \mathcal{E}_h(\bar{\Omega}_i),
\end{align*}
where $\mathcal{E}_h(\bar{\Omega}_i)$ denotes all the edges of $\mathcal{T}_i$ in $\bar{\Omega}_i$, and $\omega_\tau, \omega_e$ denotes the union of all the elements in $\mathcal{T}_{i}$ sharing at least a node with $\tau$ and $e$, respectively.
\end{lemma}

We define the local error estimator on each element $\tau\in \mathcal{T}$ as
\begin{equation}
\begin{split}
\bar{\eta}_\tau^2&=h_\tau^2\rho^{-1}\|f+\nabla \cdot \bm{z}_h\|_{0,\tau}^2+\sum_{e\in \partial \tau\cap \mathcal{F}_u}h_e\rho^{-1}\|\jump{\bm{z}_h\cdot \bm{n}}\|_{0,e}^2+\|\rho^{-\frac{1}{2}}\bm{z}_h-\rho^{\frac{1}{2}}\nabla u_h\|_{0,\tau}^2\\
&\;+\sum_{e\in \partial \tau \cap \Gamma}h_e\rho^{-1}\|\lambda_h\bm{n}_i\cdot \bm{n}_{ij}-\bm{z}_h\cdot \bm{n}_i\|_{0,e}^2+\sum_{e\in \partial \tau \cap (\mathcal{T}^{\Gamma, h}\cup \mathcal{F}_p)}h_e^{-1}\rho\|\jump{u_h}\|_{0,e}^2.
\end{split}
\label{eta-local-flux}
\end{equation}
Then, the global error estimator in energy error can be defined as
\begin{align}
\eta_2^2=\sum_{\tau\in \mathcal{T}}\bar{\eta}_\tau^2.\label{eta-2}
\end{align}

\begin{theorem}
There exists a positive constant $C$ such that the following estimate holds
\begin{align*}
\|\rho^{\frac{1}{2}}\nabla (u-u_h)\|_{0}\leq C \eta_2.
\end{align*}

\end{theorem}

\begin{proof}

We first define a function $s\in H^1_0(\Omega)$ such that
\begin{align}
(\rho \nabla s,\nabla v)=(\rho\nabla u_h, \nabla v) \quad \forall v\in H^1_0(\Omega),\label{eq:s}
\end{align}
where the existence and uniqueness of $s$ follow from Riesz representation theorem.

By taking $v=u-s$ in \eqref{eq:s}, we can get
\begin{align}
\|\rho^{\frac{1}{2}}\nabla (u-u_h)\|_0^2=\|\rho^{\frac{1}{2}}\nabla (u-s)\|_0^2+\|\rho^{\frac{1}{2}}\nabla (s-u_h)\|_0^2.\label{eq:errorde-composition}
\end{align}
We can first bound the second term by employing Lemma~\ref{lemma:nonconforming} yielding
\begin{align}
\|\rho^{\frac{1}{2}}\nabla (s-u_h)\|_0\leq C \Big(\sum_{\mathcal{F}_p\cup \mathcal{T}^{\Gamma,h} } h_e^{-1}\|\jump{\rho^{\frac{1}{2}}u_h}\|_{0,e}^2\Big)^{1/2}.\label{eq:term-nonconforming}
\end{align}
On the other hand, we have from the definition
\begin{align*}
\|\rho^{\frac{1}{2}}\nabla (u-s)\|_0&=\sup_{\psi\in H^1_0(\Omega)}\frac{(\rho \nabla (u-s), \nabla \psi)}{\|\rho^{1/2}\nabla \psi\|_0}\\
&=\sup_{\psi\in H^1_0(\Omega)}\frac{(\rho \nabla (u-u_h), \nabla \psi)}{\|\rho^{1/2}\nabla \psi\|_0}.
\end{align*}
Integration by parts implies
\begin{equation}
\begin{split}
(\rho\nabla (u-u_h), \nabla \psi)&=(\rho\nabla (u-u_h), \nabla (\psi-\Pi_h\psi))+(\rho\nabla (u-u_h), \nabla \Pi_h\psi)\\
&=\sum_{i=1}^N(\rho\nabla u\cdot \bm{n}, \psi-\Pi_h\psi)_{\Gamma_i}-(\rho\Delta u, \psi-\Pi_h\psi)-(\bm{z}_h, \nabla (\psi-\Pi_h\psi))\\
&\;+(\rho^{-\frac{1}{2}}\bm{z}_h-\rho^{\frac{1}{2}}\nabla u_h, \rho^{\frac{1}{2}}\nabla (\psi-\Pi_h\psi))+(\rho\nabla (u-u_h), \nabla \Pi_h\psi).
\end{split}
\label{eq:u-uh}
\end{equation}
The penultimate term of \eqref{eq:u-uh} can be estimated by integration by parts
\begin{align*}
(\bm{z}_h, \nabla (\psi-\Pi_h\psi))
&=\sum_{e\in \mathcal{F}_u^0}(\jump{\bm{z}_h\cdot \bm{n}}, \psi-\Pi_h\psi)_{e}+\sum_{i=1}^N(\bm{z}_h\cdot \bm{n}_i, \psi-\Pi_h\psi)_{\Gamma_i}-(\nabla \cdot \bm{z}_h, \psi-\Pi_h\psi).
\end{align*}
The last term of \eqref{eq:u-uh} can be recast into the following form by exploiting integration by parts and the second equation of \eqref{discrete}
\begin{align*}
(\rho\nabla (u-u_h), \nabla \Pi_h\psi)&=\sum_{i=1}^N(\rho\nabla u\cdot \bm{n}_i, \Pi_h\psi)_{\Gamma_{i}}-(\rho\Delta u, \Pi_h \psi)-(\rho\nabla u_h,\nabla \Pi_h\psi)\\
&=\sum_{i=1}^N(\rho\nabla u\cdot \bm{n}_i, \Pi_h\psi)_{\Gamma_{i}}+(f, \Pi_h\psi)-(\rho\nabla u_h,\nabla \Pi_h\psi)\\
&=\sum_{i=1}^N(\rho\nabla u\cdot \bm{n}_i, \Pi_h\psi)_{\Gamma_{i}}+\sum_{i=1}^N(b_i(\bm{z}_h, \Pi_h\psi)-c_i(\Pi_h\psi,\lambda_h))-(\rho\nabla u_h, \nabla \Pi_h\psi)\\
&=\sum_{i=1}^N(\rho\nabla u\cdot \bm{n}_i, \Pi_h\psi)_{\Gamma_{i}}+(\rho^{-\frac{1}{2}}\bm{z}_h-\rho^{\frac{1}{2}}\nabla u_h, \rho^{\frac{1}{2}}\nabla \Pi_h\psi)\\
&\;+\sum_{i=1}^N(\lambda_h\bm{n}_i\cdot\bm{n}_{ij}, \psi-\Pi_h\psi)_{\Gamma_i}.
\end{align*}
Finally, we can obtain by combing the above equations
\begin{align*}
(\rho\nabla (u-u_h), \nabla \psi)&=(\rho^{-\frac{1}{2}}(f+\nabla \cdot \bm{z}_h),\rho^{\frac{1}{2}}(\psi-\Pi_h\psi))+(\rho^{-\frac{1}{2}}\bm{z}_h-\rho^{\frac{1}{2}}\nabla u_h, \rho^{\frac{1}{2}}\nabla \psi)\\
&\;
+\sum_{i=1}^N(\rho^{-\frac{1}{2}}(\lambda_h\bm{n}_i\cdot\bm{n}_{ij}-\bm{z}_h\cdot \bm{n}_i),\rho^{\frac{1}{2}}(\psi-\Pi_h\psi))_{\Gamma_i}\\
&\;-\sum_{e\in \mathcal{F}_{u}^0}(\jump{\rho^{-\frac{1}{2}}\bm{z}_h\cdot \bm{n}},\rho^{\frac{1}{2}}(\psi-\Pi_h\psi))_e,
\end{align*}
which, coupling with \eqref{eq:errorde-composition}, \eqref{eq:term-nonconforming} and Lemma~\ref{lemma:error-Ih} yields the desired estimate.

\end{proof}

\section{Efficiency}\label{Sec:efficiency}
This section is devoted to establishing the lower bounds on the errors. To this end, we set the element bubble function for each element $\tau$ as $\psi_\tau$ and the edge bubble function for each edge $e$ as $\psi_e$. The properties of the bubble functions are given in the next lemma.
\begin{lemma}\label{bubblel}
The following inequalities hold for all functions $v\in P^k(\tau)$.
\begin{align}
&\|v\|_{0,\tau}\leq C \|\psi_{\tau}^{1/2}v\|_{0,\tau}\leq C \|v\|_{0,\tau},\label{bubble1}\\
&\|\nabla (\psi_{\tau} v)\|_{0,\tau}\leq C h_{\tau}^{-1} \|v\|_{0,\tau}\label{bubble3}.
\end{align}
\end{lemma}
Moreover, there exists an extension operator $P_e$ that extends any function defined on $e\in \mathcal{F}$ to the element $\tau$ and satisfies
\begin{equation}
h_e^{1/2}\|\varphi\|_{0,e}\leq C \|\psi_e P_e \varphi\|_{0,\tau}\leq C h_e^{1/2} \|\varphi\|_{0,e} \quad\forall \varphi\in P^k(e).\label{extension}
\end{equation}

Then, the lower bounds on the errors can be stated in the next theorem.
\begin{theorem}
Let $\mathcal{T}_h$ be shape regular and let $f_h$ be the piecewise linear polynomial approximation of $f$. Then, there exists a positive constant $C$ independent of $h$ such that
\begin{equation}
\begin{split}
\eta_1^2&\leq C\Big(  \|u-u_h\|_{0}^2+\sum_{\tau\in \mathcal{T}}h_\tau^2\|\nabla (u-u_h)\|_{0,\tau}^2+\sum_{e\in \mathcal{T}^{\Gamma,h}}h_e^3\|\lambda-\lambda_h\|_{0,e}^2\\
&+\sum_{\tau\in \mathcal{T}}(h_\tau^4\|\nabla (\bm{z}-\bm{z}_h)\|_{0,\tau}^2+
h_\tau^2\|\bm{z}-\bm{z}_h\|_{0,\tau}^2
+h_\tau^4\|f-f_h\|_{0,\tau}^2)\Big)
\end{split}
\label{eff-eta1}
\end{equation}
and
\begin{equation*}
\begin{split}
\eta_2^2&\leq C\Big( \|\rho^{-\frac{1}{2}}(\bm{z}-\bm{z}_h)\|_0^2+\sum_{\tau\in \mathcal{T}}\Big(h_\tau^2\|\rho^{-\frac{1}{2}}\nabla (\bm{z}-\bm{z}_h)\|_{0,\tau}^2+h_\tau^2\|\rho^{-\frac{1}{2}}(f-f_h)\|_{0,\tau}^2\\
&\;+h_\tau^{-2}\|\rho^{\frac{1}{2}}(u-u_h)\|_{0,\tau}^2
+\|\rho^{\frac{1}{2}}\nabla(u-u_h)\|_{0,\tau}^2\Big)+
\sum_{e\in\mathcal{T}^{\Gamma,h}}
h_e\|\rho^{-\frac{1}{2}}(\lambda-\lambda_h)\|_{0,e}^2\Big).
\end{split}
\end{equation*}
In addition, the following local bounds hold for any $\tau\in \mathcal{T}_h$,$e\in \partial \tau$ and $\hat{e}\in \mathcal{F}_p\cup \mathcal{T}^{\Gamma,h}$
\begin{align*}
h_\tau^2\|f+\nabla \cdot\bm{z}_h\|_{0,\tau}^2+\|\bm{z}_h-\nabla u_h\|_{0,\tau}^2&\leq C\Big(\|\bm{z}-\bm{z}_h\|_{0,\tau}^2+\|\nabla (u-u_h)\|_{0,\tau}^2+h_\tau^2\|f-f_h\|_{0,\tau}^2\Big),\\
h_e^{1/2}\|\jump{\bm{z}_h\cdot \bm{n}}\|_{0,e}&\leq C\Big( \|\bm{z}-\bm{z}_h\|_{0,D(e)}+(\sum_{\tau\in D(e)}h_\tau^2\|f-f_h\|_{0,\tau}^2)^{1/2}\Big),\\
h_e^{1/2}\|\lambda_h\bm{n}_i\cdot\bm{n}_{ij}-\bm{z}_h\cdot\bm{n}_i\|_{0,e}&\leq C\Big(h_e^{1/2}\|\lambda-\lambda_h\|_{0,e}+\|\bm{z}-\bm{z}_h\|_{0,\tau}+h_\tau \|\nabla (\bm{z}-\bm{z}_h)\|_{0,\tau}\Big)
\end{align*}
and
\begin{align*}
\|\jump{u_h}\|_{0,\hat{e}}\leq C \sum_{\tau\in D(\hat{e})}\Big( h_\tau^{-1/2}\|u-u_h\|_{0,\tau}+h_\tau^{1/2}\|\nabla (u-u_h)\|_{0,\tau}\Big).
\end{align*}

\end{theorem}


\begin{proof}
Let $R_\tau(f)=f+\nabla \cdot \bm{z}_h$. Green's theorem and the Cauchy-Schwarz inequality imply
\begin{align*}
(R_\tau(f_h), \psi_\tau R_\tau(f_h))_{\tau}&=(-\nabla \cdot (\bm{z}-\bm{z}_h), \psi_\tau R_\tau(f_h))_\tau+(f_h-f, \psi_\tau R_\tau(f_h))_\tau\\
&=(\bm{z}-\bm{z}_h,
\nabla( \psi_\tau R_\tau(f_h)))_\tau+(f_h-f, \psi_\tau R_\tau(f_h))_\tau\\
&\leq C ( \|\bm{z}-\bm{z}_h\|_{0,\tau}
\|\nabla (\psi_\tau R_\tau(f_h))\|_{0,\tau}+\|f-f_h\|_{0,\tau}\|\psi_\tau R_\tau(f_h)\|_{0,\tau})\\
&\leq C (\|\bm{z}-\bm{z}_h\|_{0,\tau}
h_{\tau}^{-1}\|R_\tau(f_h)\|_{0,\tau}
+\|f-f_h\|_{0,\tau}\| R_\tau(f_h)\|_{0,\tau}),
\end{align*}
where in the last inequality, we use \eqref{bubble1} and \eqref{bubble3}.

Combining the above inequality with \eqref{bubble1}, we can achieve
\begin{eqnarray*}
\|R_\tau(f_h)\|_{0,\tau}^2\leq C ( \|\bm{z}-\bm{z}_h\|_{0,\tau}h_\tau^{-1} \|R_\tau(f_h)\|_{0,\tau}+\|f-f_h\|_{0,\tau} \|R_\tau(f_h)\|_{0,\tau}),
\end{eqnarray*}
which yields
\begin{eqnarray}
h_\tau \|R_\tau(f)\|_{0,\tau}\leq C(\|
\bm{z}-\bm{z}_h\|_{0,\tau}+h_\tau \|f-f_h\|_{0,\tau}).\label{eff-fh}
\end{eqnarray}
Next, fix an edge $e\in \mathcal{F}_{u}^0$, for any $w=\psi_eP_e\jump{\bm{z}_h\cdot \bm{n}}\in H^1_0(D(e))$, we have from integration by parts
\begin{align*}
(\jump{\bm{z}_h}\cdot \bm{n}_e, \psi_eP_e\jump{\bm{z}_h\cdot \bm{n}})_e&=\sum_{\tau\in D(e)}((\bm{z}_h-\bm{z})\cdot \bm{n}_\tau,\psi_eP_e\jump{\bm{z}_h\cdot \bm{n}} )_{\partial \tau}\\
&= \sum_{\tau\in D(e)} \Big((\nabla \cdot (\bm{z}_h-\bm{z}),  \psi_eP_e\jump{\bm{z}_h\cdot \bm{n}})_\tau+(\bm{z}_h-\bm{z}, \nabla  (\psi_eP_e\jump{\bm{z}_h\cdot \bm{n}}))_\tau\Big)\\
&= \sum_{\tau\in D(e)} ((R_\tau(f), \psi_eP_e\jump{\bm{z}_h\cdot \bm{n}})_{\tau}+(\bm{z}_h-\bm{z}, \nabla  (\psi_eP_e\jump{\bm{z}_h\cdot \bm{n}}))_{\tau}),
\end{align*}
which, coupling with \eqref{extension}, inverse inequality and \eqref{eff-fh} yields
\begin{equation*}
\begin{split}
\|\jump{\bm{z}_h\cdot \bm{n}}\|_{0,e}^2&\leq C \sum_{\tau\in D(e)}  \Big(\|\bm{z}-\bm{z}_h\|_{0,\tau}
\|\nabla (\psi_eP_e\jump{\bm{z}_h\cdot \bm{n}})\|_{0,\tau}+\|f+\nabla \cdot \bm{z}_h\|_{0,\tau}\| \psi_eP_e\jump{\bm{z}_h\cdot \bm{n}}\|_{0,\tau}\Big)\\
&\leq C  \sum_{\tau\in D(e)} \Big(\|\bm{z}-\bm{z}_h\|_{0,\tau}
h_{\tau}^{-1}
\|\psi_eP_e\jump{\bm{z}_h\cdot \bm{n}}\|_{0,\tau}+\|f+\nabla \cdot \bm{z}_h\|_{0,\tau}\| \psi_eP_e\jump{\bm{z}_h\cdot \bm{n}}\|_{0,\tau}\Big)\\
&\leq C \Big(\Big(\sum_{\tau\in D(e)}\|\bm{z}-\bm{z}_h\|_{0,\tau}^2
\Big)^{1/2}+\Big(\sum_{\tau\in D(e)}h_\tau^2
\|f-f_h\|_{0,\tau}^2\Big)^{\frac{1}{2}}\Big)  h_e^{-\frac{1}{2}}\| \jump{\bm{z}_h\cdot \bm{n}}\|_{0,e}.
\end{split}
\end{equation*}
Thus
\begin{align*}
h_e^{1/2}\|\jump{\bm{z}_h\cdot \bm{n}}\|_{0,e}\leq C( \|\bm{z}-\bm{z}_h\|_{0,D(e)}+(\sum_{\tau\in D(e)}h_\tau^2\|f-f_h\|_{0,\tau}^2)^{1/2}).
\end{align*}
Triangle inequality implies
\begin{align*}
\|\bm{z}_h-\nabla u_h\|_0\leq \|\bm{z}-\bm{z}_h\|_0+\|\nabla (u-u_h)\|_0.
\end{align*}
Finally, the triangle inequality and trace inequality \eqref{trace} yield
\begin{align*}
\|\lambda_h\bm{n}_i\cdot\bm{n}_{ij}-\bm{z}_h\cdot \bm{n}_i\|_{0,e}&\leq C(\|\lambda-\lambda_h\|_{0,e}+\|\bm{z}\cdot \bm{n}-\bm{z}_h\cdot \bm{n}\|_{0,e})\\
&\leq C\Big( \|\lambda-\lambda_h\|_{0,e}+h_\tau^{-1/2}
\|\bm{z}-\bm{z}_h\|_{0,\tau}
+h_\tau^{1/2}\|\nabla (\bm{z}-\bm{z}_h)\|_{0,\tau}\Big)
\end{align*}
and
\begin{align*}
\|\jump{u_h}\|_{0,\hat{e}}=\|\jump{u-u_h}\|_{0,\hat{e}}\leq C\sum_{\tau\in D(\hat{e})}\Big( h_\tau^{-1/2}\|u-u_h\|_{0,\tau}+h_\tau^{1/2}\|\nabla (u-u_h)\|_{0,\tau}\Big).
\end{align*}
The preceding arguments complete the assertion.
\end{proof}

\begin{remark}{\rm It follows from Lemma~\ref{priori-estimate} that the orders of convergence for all the terms present in the right hand side of \eqref{eff-eta1} are comparable to $\|u-u_h\|_0$. Thus, this bound, combined with Theorem~\ref{thm:eta1}, implies that $\eta_1$ is an efficient and reliable estimator for the potential $L^2$ error. Similarly, $\eta_2$ is also an efficient and reliable estimator for the energy error.}
\end{remark}

\section{Numerical experiments}\label{Sec:numerical}
In this section we present several numerical experiments to demonstrate the performance of the proposed error estimators.
 The adaptive mesh pattern and convergence history are reported for each example. In all of our simulations, we use piecewise linear elements, i.e., $k=1$. Since the new triangulation $\mathcal{T}$ is only formed to define the method and it is not
 a refinement. Therefore, in our refinement algorithm, we will carry out the refinement on $\mathcal{T}_h$ by an estimator defined on each $\rho\in \mathcal{T}_h$. We define the error estimator as
\begin{align*}
\xi_{\rho}=\sum_{\tau\in \mathcal{T}, \tau\cap \rho\neq \emptyset}\eta_{\tau}^2.
\end{align*}
Moreover, for any subset $\mathcal{M}\subset \mathcal{T}_h$, we define
\begin{align*}
\xi^2(\mathcal{M}):=\sum_{\rho\in \mathcal{M}}\xi^2(\rho).
\end{align*}
Similar definitions can be applied for $\bar{\eta}_\tau$ defined in \eqref{eta-local-flux}.

Our adaptive refinement can be implemented by the following iteration:
\begin{enumerate}
  \item Start with an initial mesh $\mathcal{T}_h^0$.
  \item Solve the discrete problem \eqref{discrete} for $(\bm{u}_h^\ell, \bm{z}_h^\ell, p_h^\ell)$ with respect to $\mathcal{T}_h^\ell$.
  \item Compute $\xi_\rho, \forall \rho\in \mathcal{T}_h^\ell$.
  \item Mark the minimal set $\mathcal{M}\subseteq \mathcal{T}_h$ satisfying $\theta \sum_{\rho\in \mathcal{T}_{h}}\xi_{\rho}^{2}\leq \sum_{\rho\in \mathcal{M}}\xi_{\rho}^{2}$ for some fixed parameter $\theta\in(0,1)$.
  \item Refine marked triangles and compute $\mathcal{T}^{\ell+1}_h$ by {\it red and green refinement} for adaptive mesh. Update $\ell$ and go to step 2.
\end{enumerate}

\begin{example}
{\rm In this example, $\Omega=(0,1)^2$ and $\rho=1$, we consider the exact solution given by
\begin{align*}
u(x,y)=1000xye^{-100(x^2+y^2)}.
\end{align*}
The global domain $\Omega$ is decomposed into four square subdomains and the initial grid in each subdomain is $2\times 2$.}
\end{example}

The contour plot of the exact solution and the adaptive mesh pattern arising from the energy error estimator are reported in
{\sc Fig}.~\ref{Ex1-mesh}. The adaptive mesh pattern for the error estimator in potential $L^2$ error is similar and is omitted for simplicity. We note that the grids are appropriately refined along the boundary layers.

The convergence history for $\|u-u_h\|_0$ and $\eta_1$ as well as $\|\nabla (u-u_h)\|_0$ and $\eta_2$ are displayed in {\sc Fig}.~\ref{Ex1-convergence}. We observe that the adaptive solution needs much fewer elements to provide the same accuracy.
\begin{figure}[H]
\centering
\scalebox{0.3}{
\includegraphics[width=20cm]{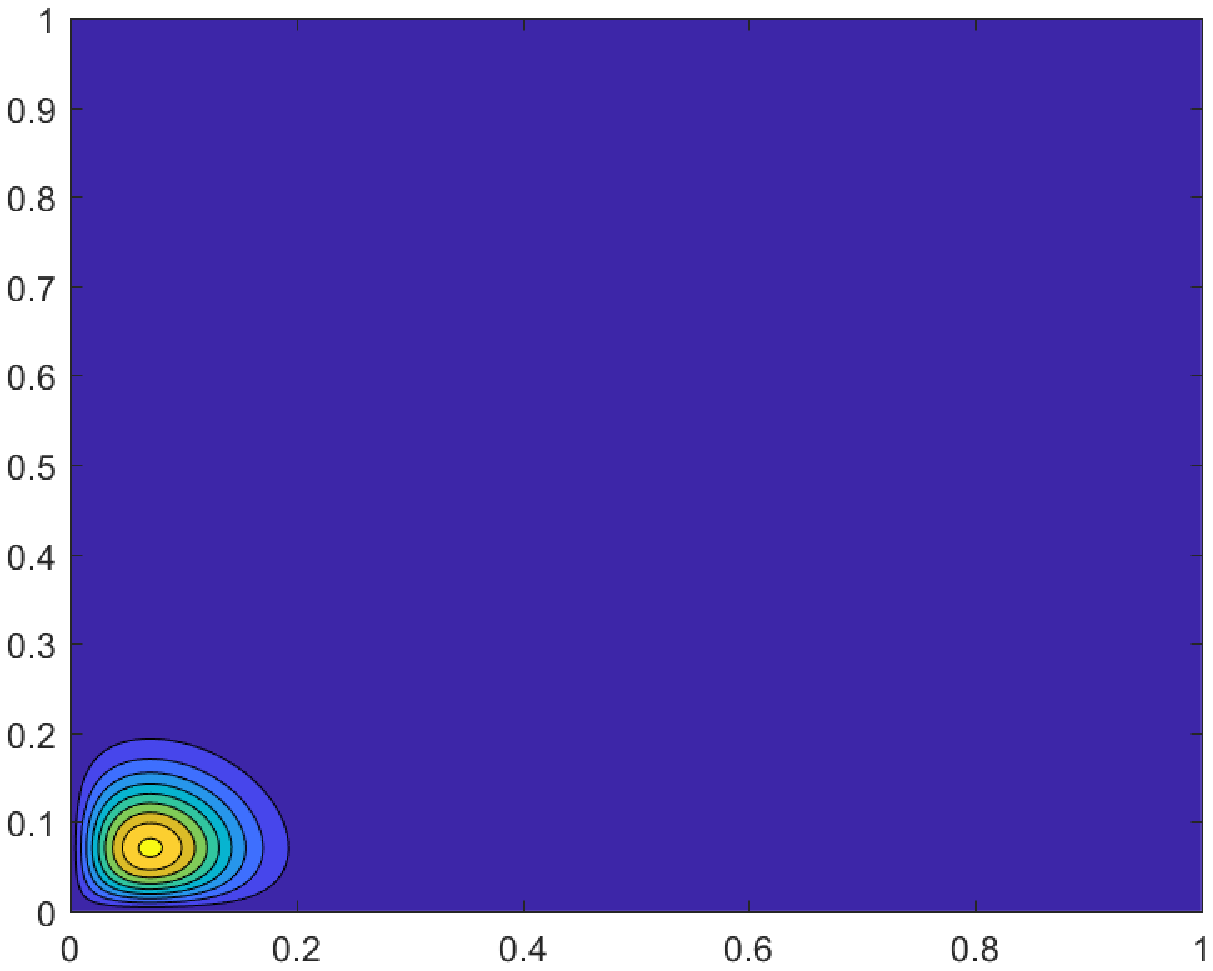}
}
\scalebox{0.3}{
\includegraphics[width=20cm]{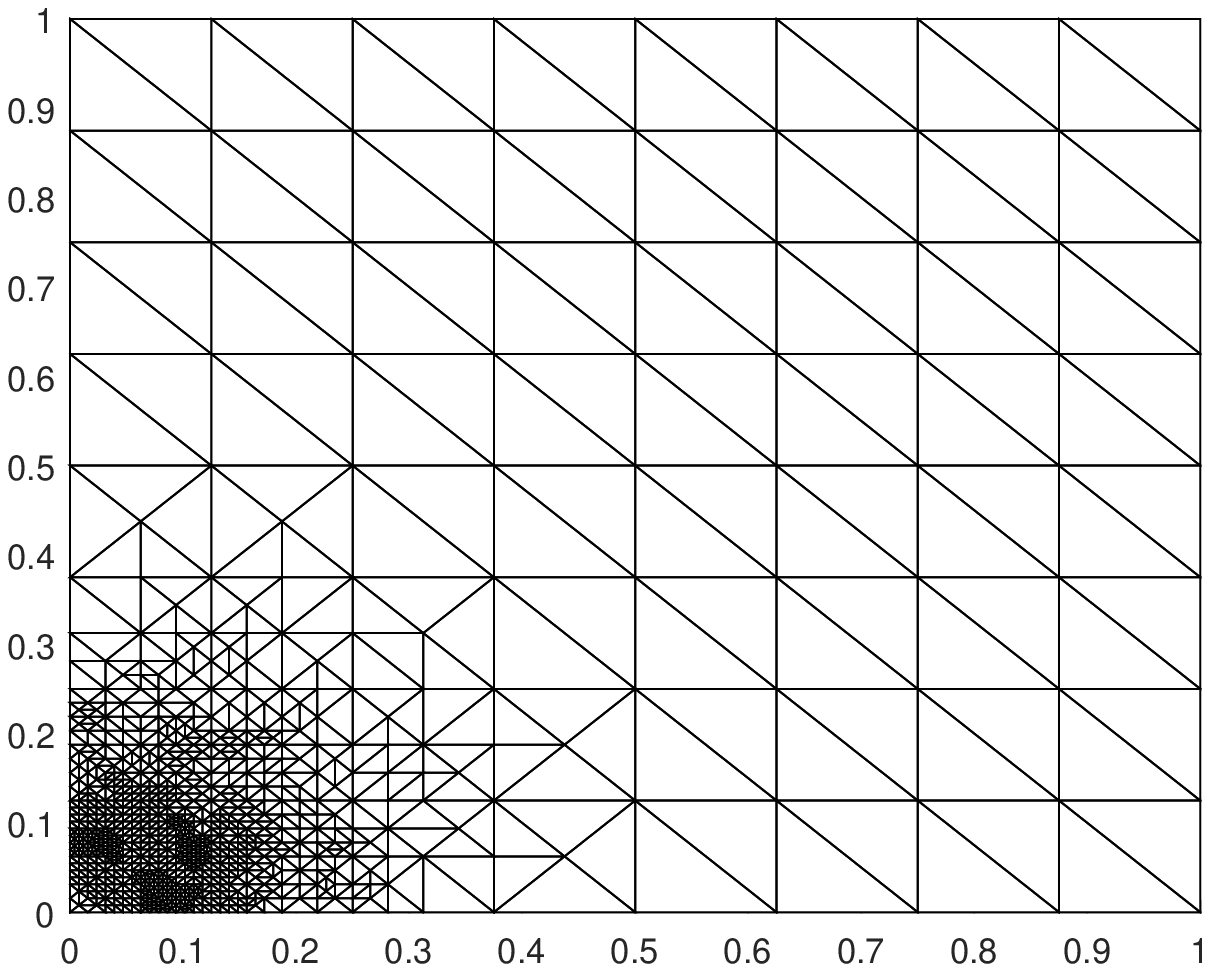}
}
\caption{The contour plot of the exact solution and the adaptive mesh pattern for example 1.}
\label{Ex1-mesh}
\end{figure}

\begin{figure}[H]
\centering
\scalebox{0.3}{
\includegraphics[width=20cm]{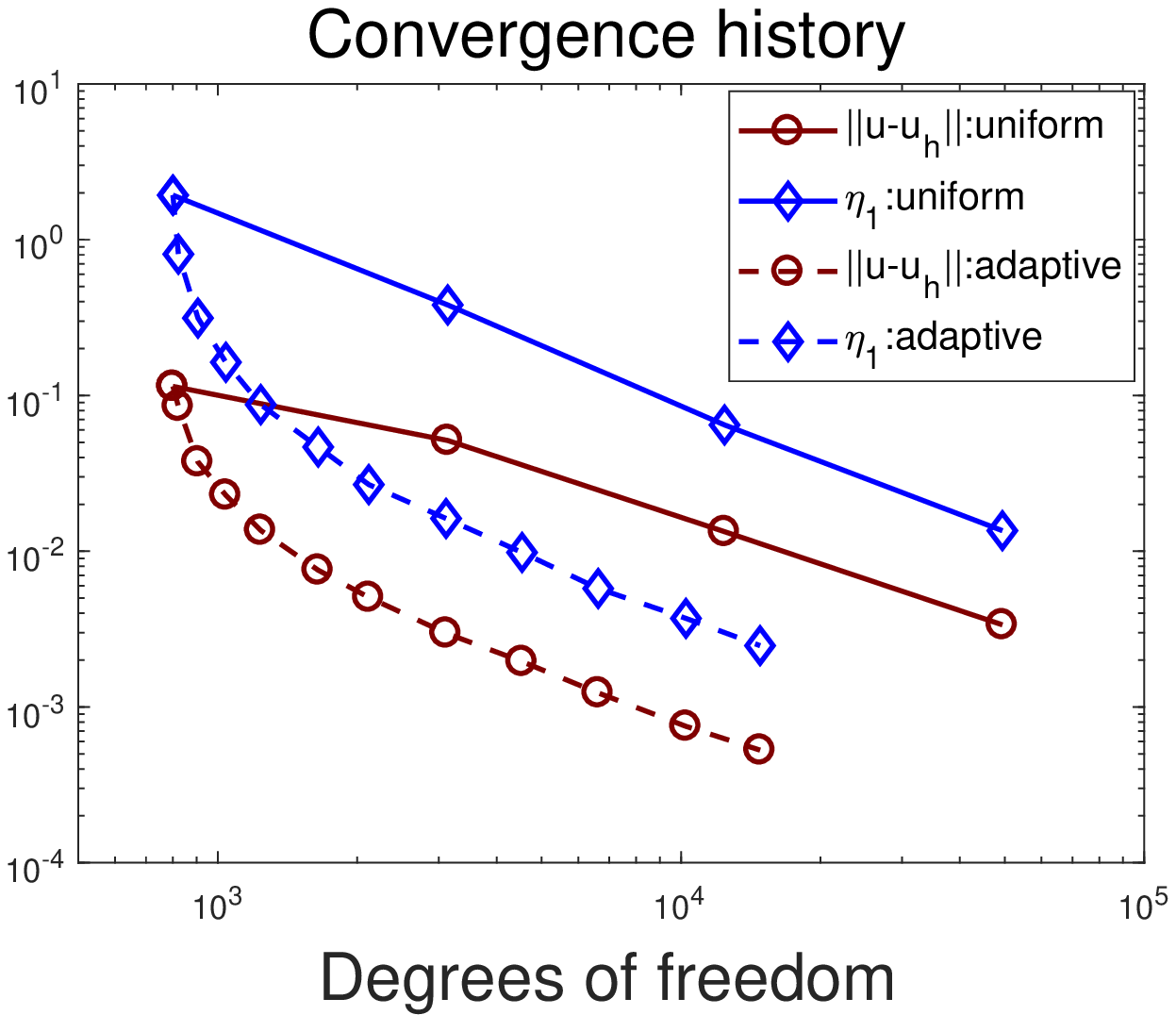}
}
\scalebox{0.3}{
\includegraphics[width=20cm]{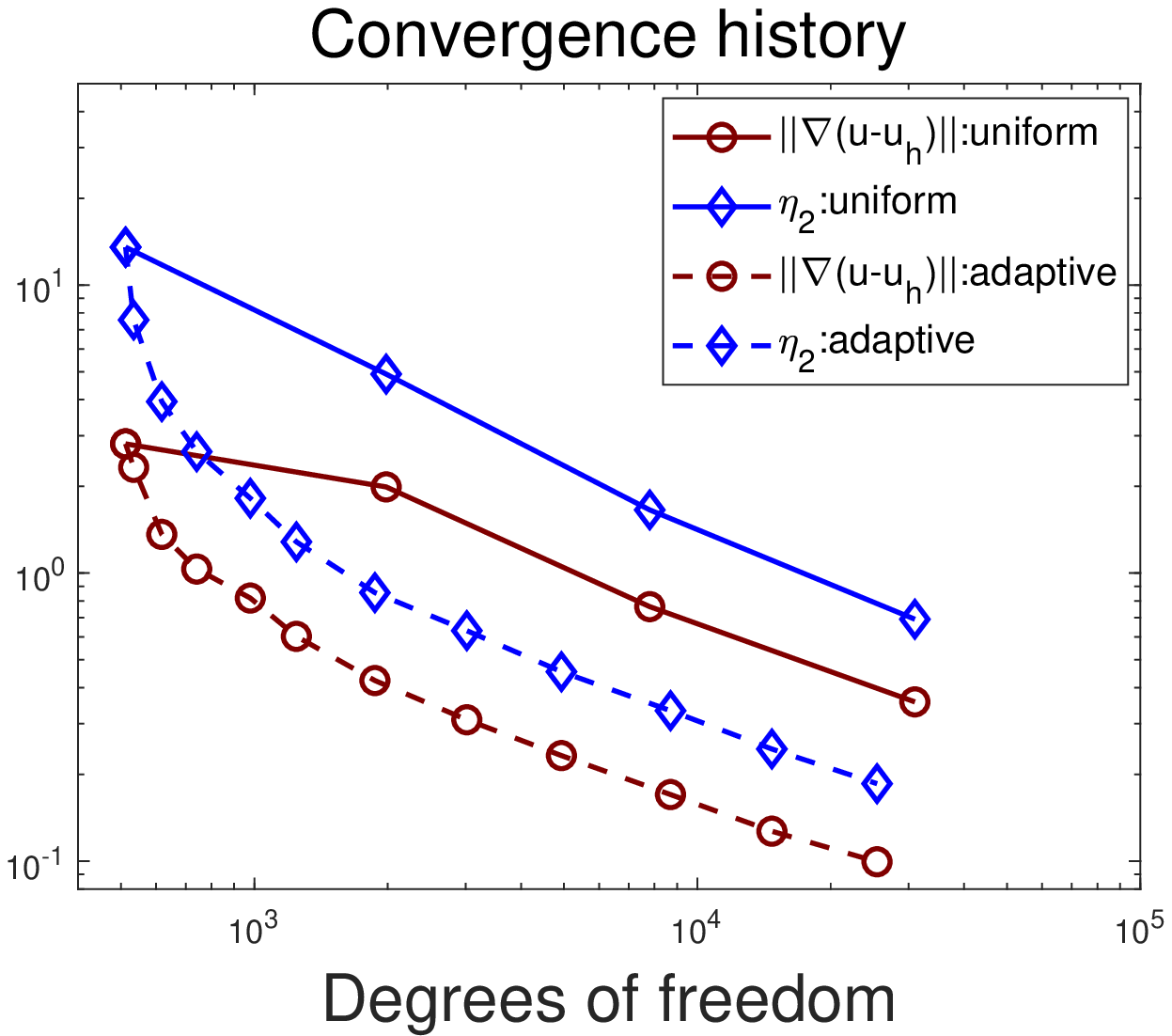}
}
\caption{Convergence history: potential error estimator (left) and energy error estimator (right).}
\label{Ex1-convergence}
\end{figure}

\begin{example}
{\rm In this example, we consider a non-smooth solution in $H^{3/2-\delta}(\Omega)$ with $\delta>0$ defined by $u(r,\theta)=r^{1/2}\cos(2\theta)$ with polar coordinates $(r,\theta)$ centered at $(0.5,0.5)$. In addition, we let $\rho=1$. We assume that the computational domain is decomposed into four square subdomains.}
\end{example}
The initial mesh and adaptive mesh pattern using the energy error estimator are shown in {\sc Fig}.~\ref{Ex2-mesh}. We observe that the singularly can be well captured by the proposed error estimators.

The convergence history for $\|u-u_h\|_0$ and $\eta_1$ as well as $\|\nabla (u-u_h)\|_0$ and $\eta_2$ under uniform refinement and adaptive refinement, respectively, are shown in {\sc Fig}.~\ref{Ex2-convergence}. It is clear that the order of convergence for $\|u-u_h\|_0$ and $\eta_1$ under uniform refinement is approximately 1.5, while the order of convergence for $\|u-u_h\|_0$ and $\eta_1$ under adaptive refinement is approximately 2. On the other hand, the order of convergence for $\|\nabla (u-u_h)\|_0$ and $\eta_2$ under uniform refinement is approximately 0.5, while the order of convergence for $\|\nabla (u-u_h)\|_0$ and $\eta_2$ under adaptive refinement is approximately 1. This demonstrates that under uniform refinement we can achieve the reduced convergence rate reflecting singularity, and optimal convergence rates can be recovered by employing adaptive mesh refinement. This example highlights that adaptive mesh refinement outperforms uniform mesh refinement and can lead to optimal convergence rate even with low solution regularity.
\begin{figure}[H]
\centering
\includegraphics[width=6cm]{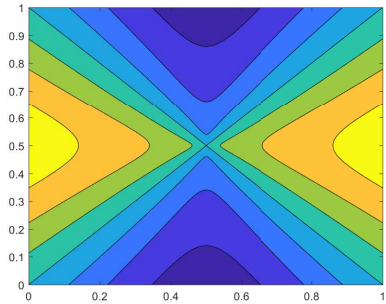}
\caption{The contour plot of the exact solution for example 2.}
\label{Ex2-exact}
\end{figure}

\begin{figure}[H]
\centering
\scalebox{0.3}{
\includegraphics[width=20cm]{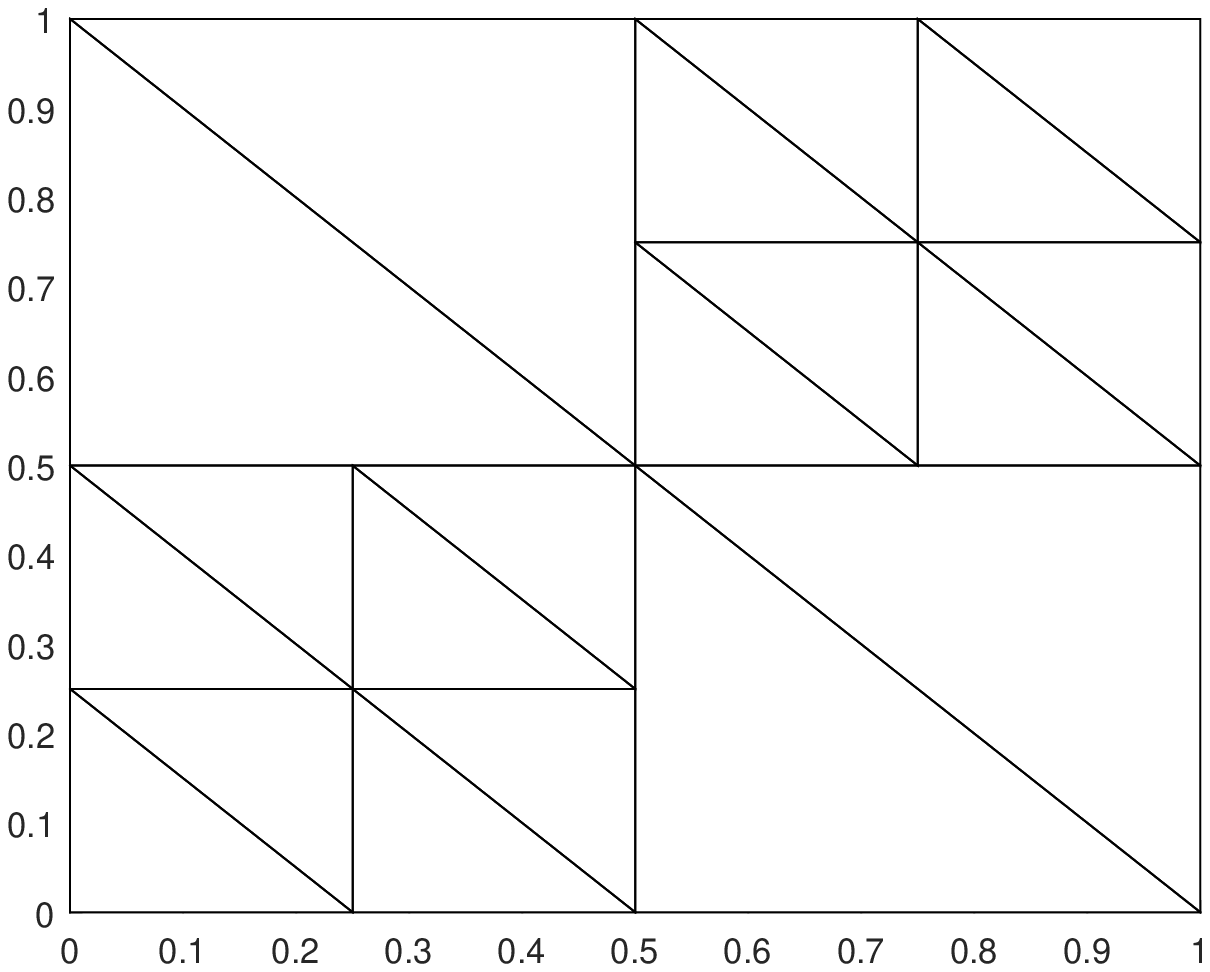}
}
\scalebox{0.3}{
\includegraphics[width=20cm]{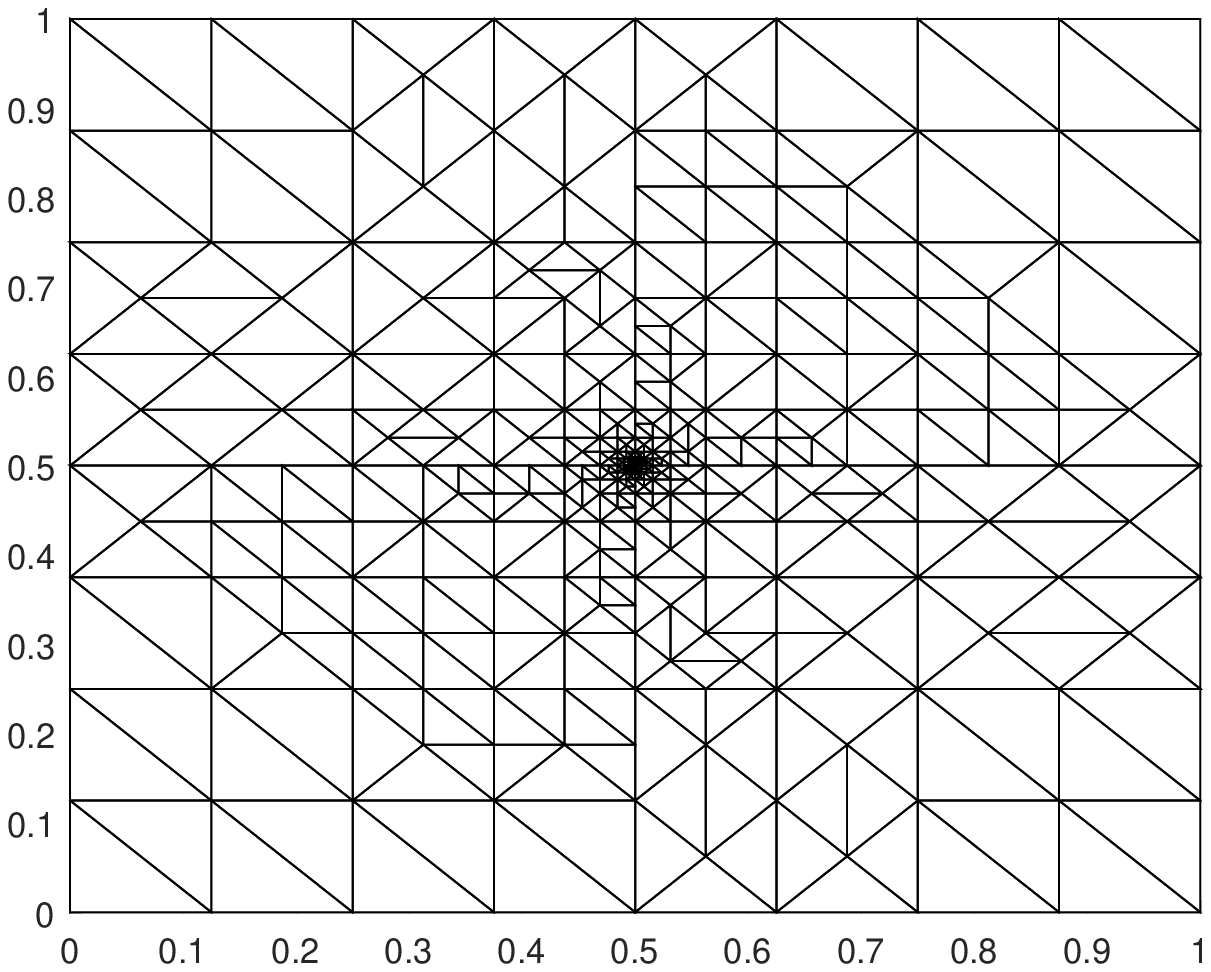}
}
\caption{Initial mesh (left) and adaptive mesh pattern (right).}
\label{Ex2-mesh}
\end{figure}

\begin{figure}[H]
\centering
\scalebox{0.3}{
\includegraphics[width=20cm]{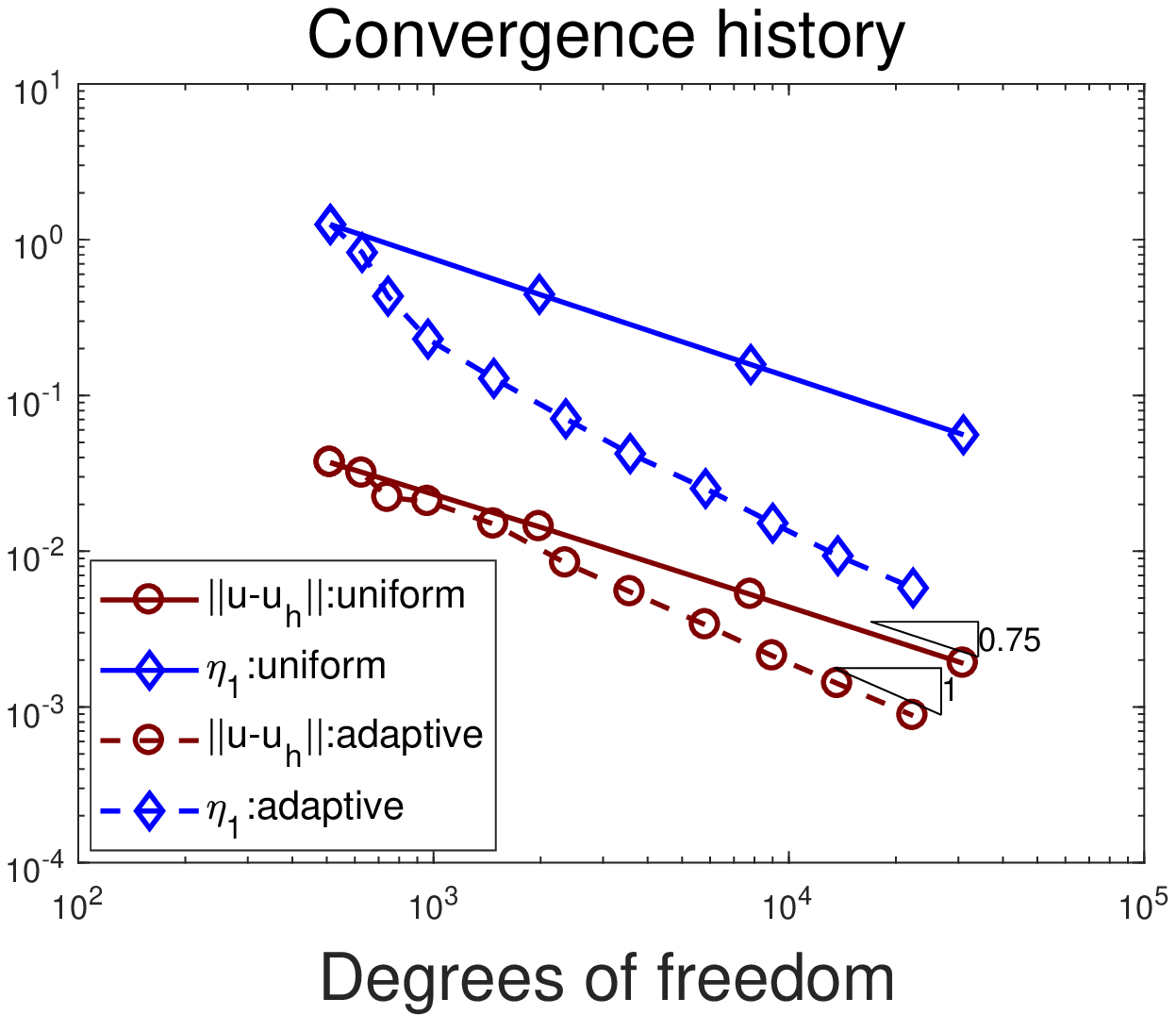}
}
\scalebox{0.3}{
\includegraphics[width=20cm]{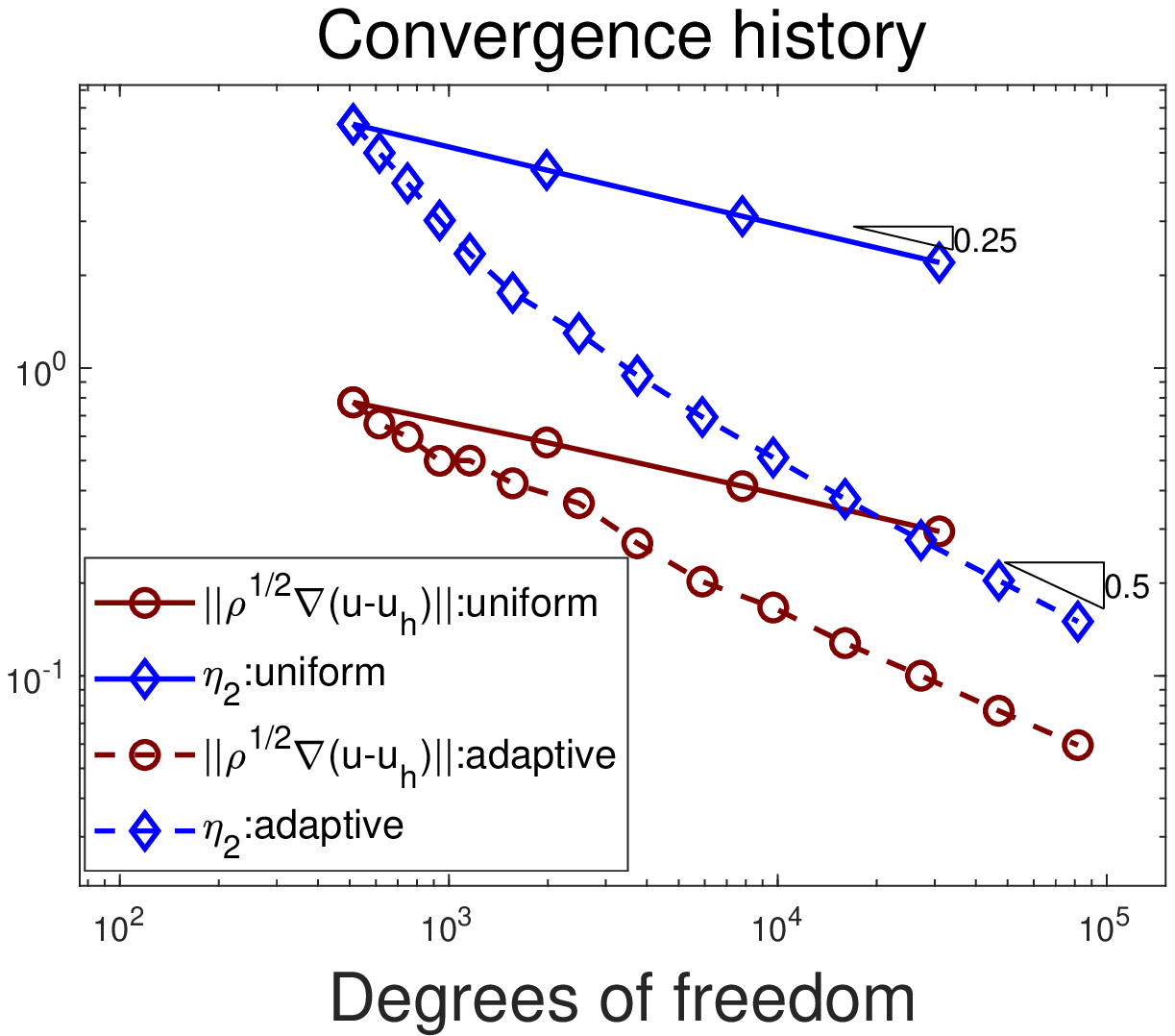}
}
\caption{Convergence history: potential error estimator (left) and energy error estimator (right).}
\label{Ex2-convergence}
\end{figure}

\begin{example}
{\rm
Our third example is an interface problem which exhibits an interface singularity, cf. \cite{rivi, martin}.
Consider $\Omega=(-1,1)\times(-1,1)$, divided into four subdomains $\Omega_i$ along the Cartesian axes (the subregion $\{x>0,y>0\}\cap \Omega$ is denoted as $\Omega_1$ and the subsequent numbering is done counterclockwise). The exact solution is given by
\begin{eqnarray}
u(r,\theta)=r^{\alpha}(K_{i}\sin(\alpha \theta)+S_{i} \cos(\alpha \theta))\nonumber
\end{eqnarray}
in each $\Omega_{i}$. The solution is continuous across the interfaces and the normal component of its flux $\bm{z}$ is continuous; it exhibits a singularity at the origin and it only belongs to $H^{1+\alpha}(\Omega)$.
We take the piecewise constant coefficient as $\rho_1=\rho_3=5, \rho_2=\rho_4=1$ and $\alpha=0.53544095$.
The values of $K_{i}, S_{i}$ can be found in, e.g., \cite{martin}.
}
\end{example}

The initial mesh is the same as {\sc Fig}.~\ref{Ex2-convergence} and the convergence history and adaptive mesh pattern are reported in {\sc Fig}.~\ref{Ex3-convergence} and {\sc Fig}.~\ref{Ex3-mesh}, respectively. Again, reduced convergence rate can be achieved for uniform refinement due to low solution regularity, while optimal convergence rates can be recovered by employing adaptive mesh refinement. In addition, the singularity is well captured. This example once again illustrates that the proposed error estimators can guide adaptive mesh refinement.

\begin{figure}[H]
\centering
\scalebox{0.3}{
\includegraphics[width=20cm]{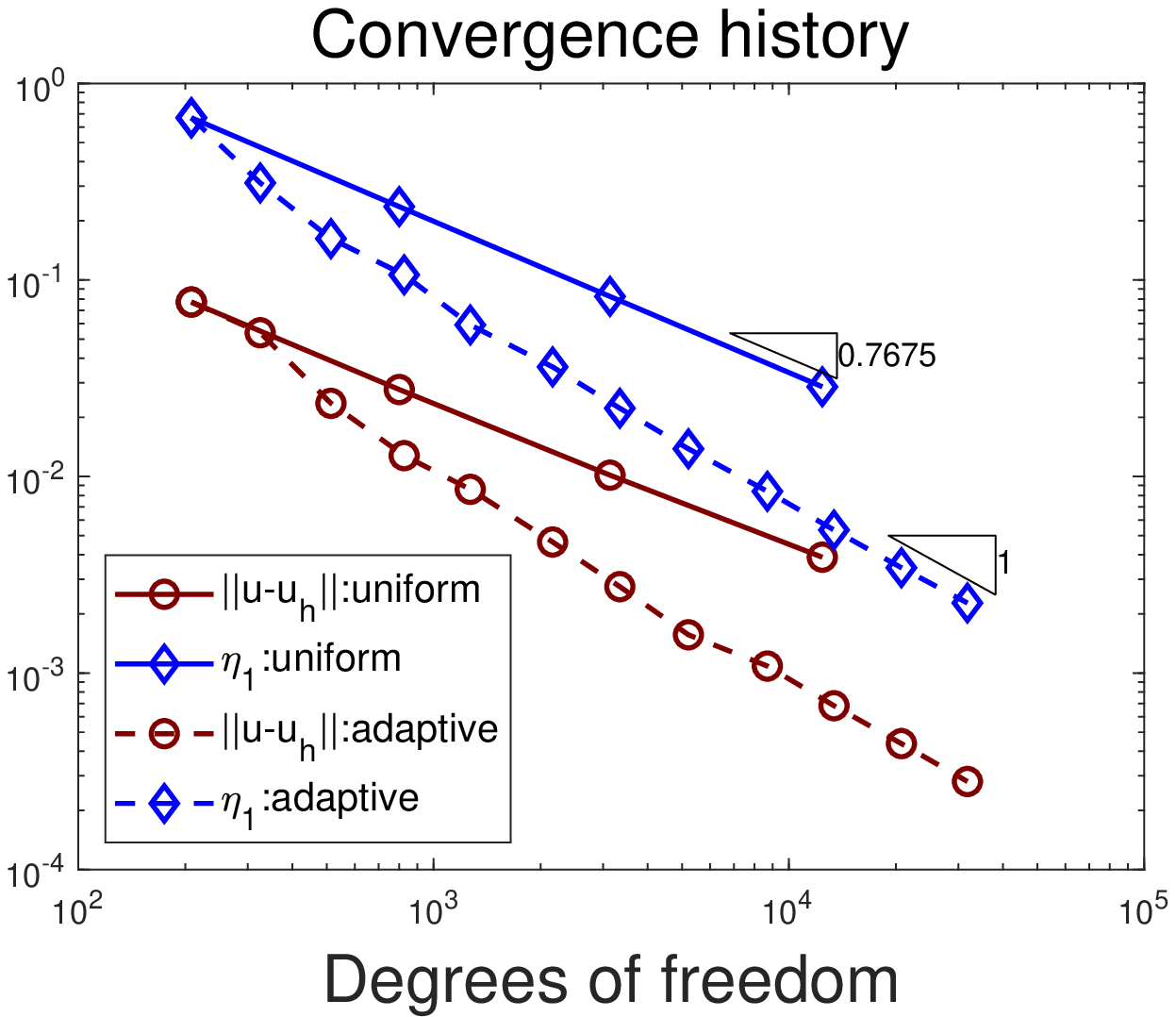}
}
\scalebox{0.3}{
\includegraphics[width=20cm]{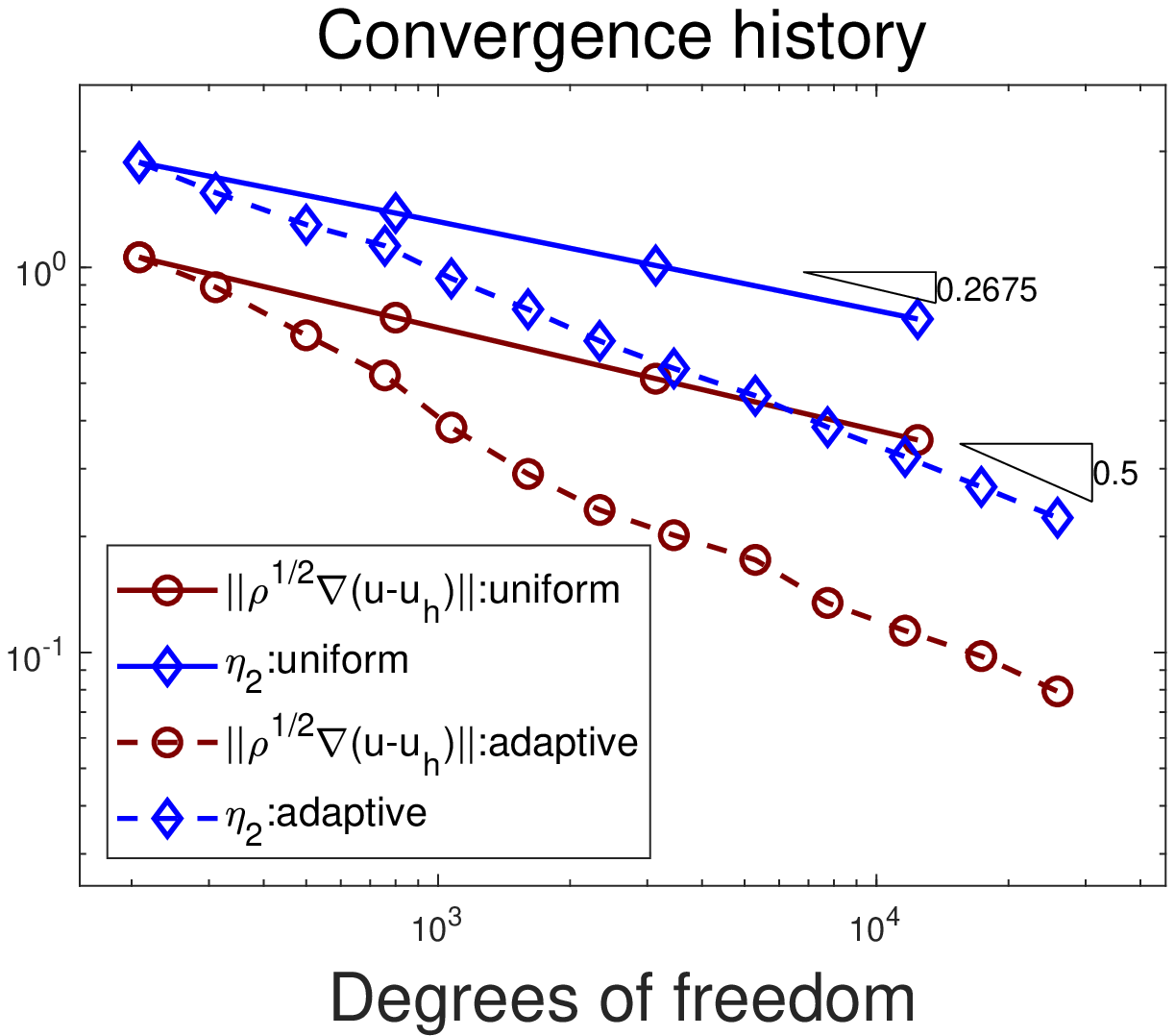}
}
\caption{Convergence history: potential error estimator (left) and energy error estimator (right).}
\label{Ex3-convergence}
\end{figure}

\begin{figure}[H]
\centering
\scalebox{0.3}{
\includegraphics[width=20cm]{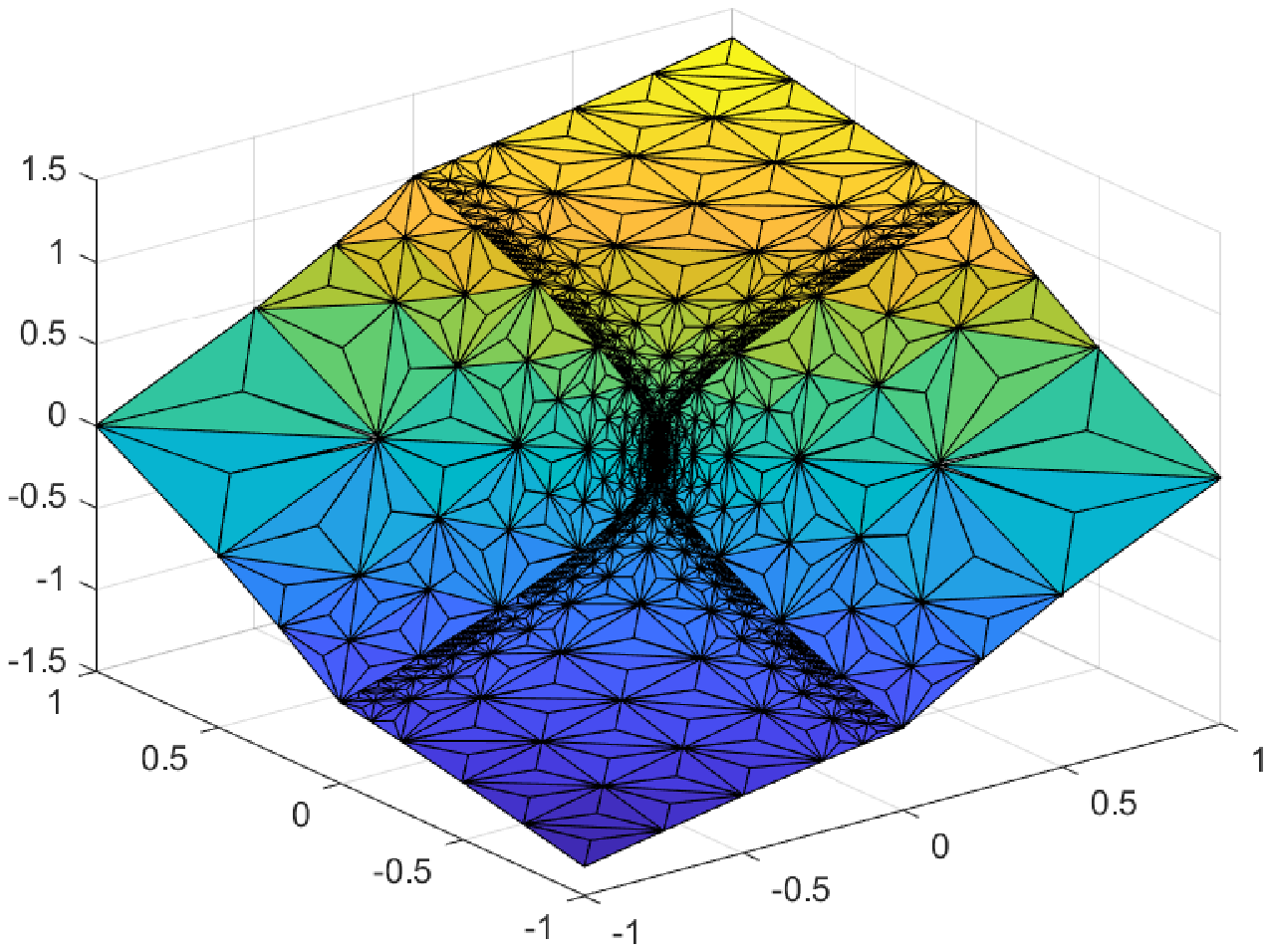}
}
\scalebox{0.3}{
\includegraphics[width=20cm]{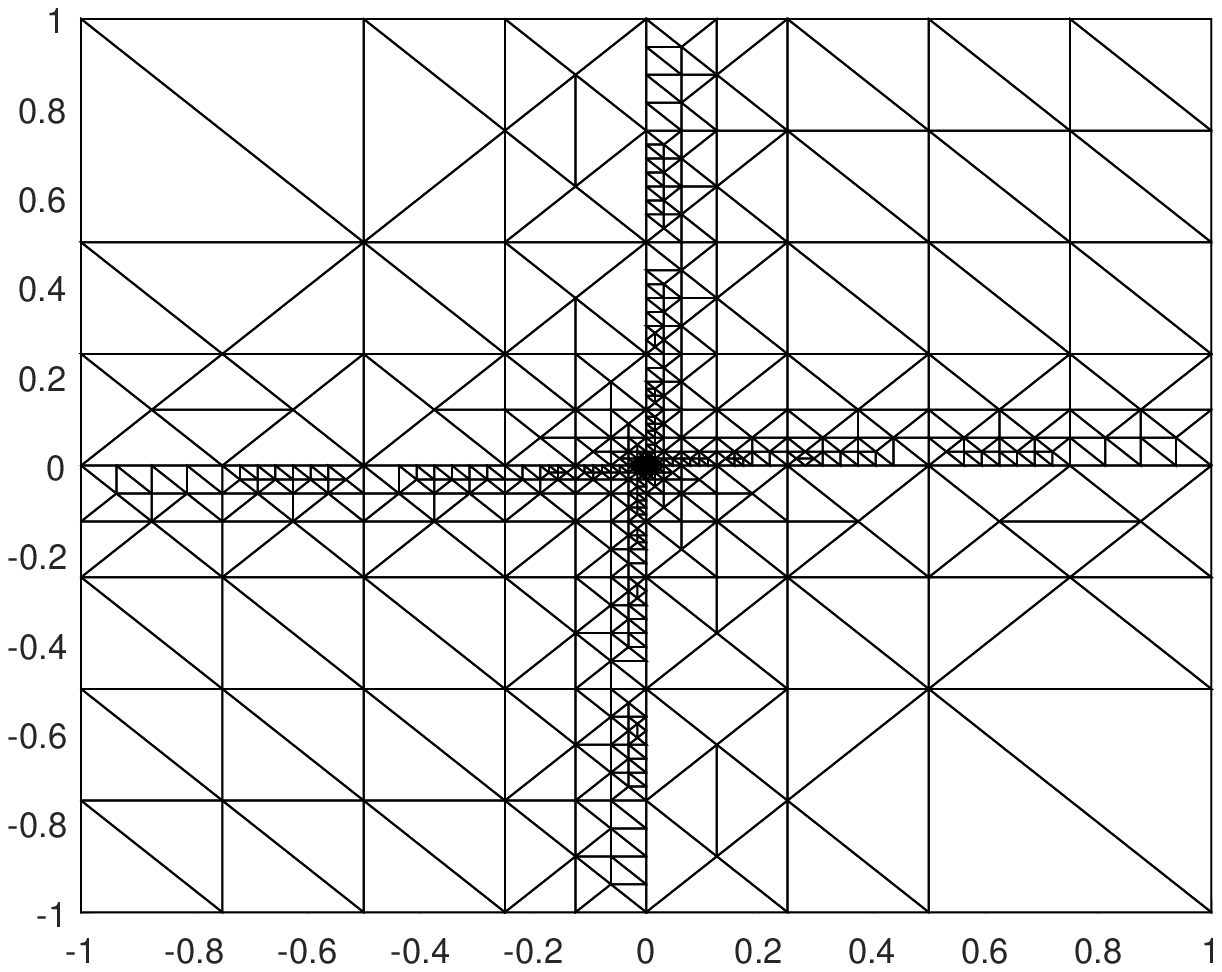}
}
\caption{Approximate solution on adaptively refined meshes (left) and adaptive mesh pattern for energy error estimator (right).}
\label{Ex3-mesh}
\end{figure}

\section{Conclusion}\label{Sec:conclusion}

In this paper, we have proposed two residual-type error estimators in potential $L^2$ error and energy error, respectively. The proposed error estimators are proved to be reliable and efficient. The key idea is to exploit the duality argument for potential $L^2$ error. To derive an error estimator in energy error, we decompose the energy error into conforming part and nonconforming part via the introduction of an auxiliary function. The numerical results demonstrate that the singularities can be well captured by the proposed error estimators, in addition, the superiority of adaptive mesh refinement over uniform mesh refinement is clearly visible in the improved convergence rate for solutions of limited regularity.

\section*{Acknowledgements}

The research of Eric Chung is partially supported by the Hong Kong RGC General Research Fund (Project numbers 14304217 and 14302018), CUHK Faculty of Science Direct Grant 2018-19 and NSFC/RGC Joint Research Scheme (Project number HKUST620/15).

\end{document}